\documentclass[11pt]{article}

\usepackage[a4paper,left=3cm,right=3cm,top=3cm,bottom=3cm,bindingoffset=5mm]{geometry}

\usepackage[utf8]{inputenc}

\usepackage{amsmath, amsthm, amssymb}
\usepackage{color,hyperref}
\usepackage[all,cmtip]{xy}

\usepackage{graphicx}

\usepackage{ mathrsfs }

\usepackage{enumerate}

\usepackage{color}

\newtheorem{mydef}{Definition}[section]
\newtheorem{thm}[mydef]{Theorem}
\newtheorem{prop}[mydef]{Proposition}

\newtheorem{cor}[mydef]{Corollary}

\theoremstyle{remark}

\newtheorem*{rem}{Remark}

\newtheorem{claim}[mydef]{Claim}
\newtheorem*{ac}{Acknowledgments}

\title{Rapoport-Zink uniformization for the moduli space of polarized K3 surfaces}
\author{Tobias Kreutz}
\date{\vspace{-5ex}}

\begin{document}
\maketitle
\begin{abstract}
We compute the image of the $p$-adic period map for polarized K3 surfaces with supersingular reduction.
This gives rise to a Rapoport-Zink type uniformization of their moduli space by an explicit open rigid analytic subvariety of a local Shimura variety of orthogonal type.
In contrast to the case of Rapoport-Zink uniformization of Shimura varieties and in analogy to the complex case, the uniformi\-zing domain does not carry an action of a $p$-adic Lie group, but only of a discrete subgroup.
We briefly sketch how the same arguments can be applied to obtain a uniformization for the moduli space of smooth cubic fourfolds with supersingular reduction.
\end{abstract}

\tableofcontents

\section*{Introduction: Uniformization of moduli spaces}
\addcontentsline{toc}{section}{Introduction: Uniformization of moduli spaces}

In this paper, we want to give new examples of $p$-adic uniformization for moduli spaces of algebraic varieties, in the case of polarized K3 surfaces and cubic fourfolds.

\subsection*{Shimura varieties of abelian type}
\addcontentsline{toc}{subsection}{Shimura varieties of abelian type}

We first recall the theory for Shimura varieties of Hodge type and abelian type, which are moduli spaces of abelian varieties (or more generally of abelian motives) with additional structure (cf. \cite{Milnemotives}).
The complex analytification of such a Shimura variety $Sh$ has a uniformization as a double quotient
\begin{equation}\label{complexShuniformization} Sh_{\mathbb{C}}^{an} = G(\mathbb{Q}) \backslash \mathcal{D} \times G(\mathbb{A}_f)/K,\end{equation} where $\mathcal{D}$ is a hermitian symmetric domain which is the $G(\mathbb{R})$-conjugacy class of a minuscule Hodge cocharacter $h: \mathbb{S} \to G_{\mathbb{R}}$ for a reductive group $G$ over $\mathbb{Q}$, and $K \subset G(\mathbb{A}_f)$ is a compact open subgroup.
The complex uniformization (\ref{complexShuniformization}) can also be written as a disjoint union $$Sh_{\mathbb{C}}^{an} = \coprod_g \Gamma_g \backslash \mathcal{D},$$ where $g$ runs through a set of representatives of the finite double coset $G(\mathbb{Q}) \backslash G(\mathbb{A}_f)/K$.
Note that the uniformization provides a purely group-theoretical description of the complex moduli space.

On the other hand, one is also interested in the arithmetic study of such a moduli space.
Namely, the moduli problem for $Sh$ is naturally defined over a number field $E$, the so-called reflex field of $Sh$, and for abelian type Shimura varieties with hyperspecial level at $p$ one can even define an integral canonical model $\mathscr{S}$ over $\mathcal{O}_{E, (v)}$, where $v$ is a prime of $E$ lying over the rational prime $p$. We refer the reader to the work of Kisin \cite{Intmod} for a construction of integral canonical models of Shimura varieties of abelian type with hyperspecial level at $p$.

In this situation, there is an analog of the complex uniformization in the $p$-adic setting, known as Rapoport-Zink uniformization.
Due to the more complicated structure of the moduli space in the $p$-adic case, there does not exist a uniformization of the entire moduli space; one first has to stratify the special fiber $\overline{\mathscr{S}} = \mathscr{S}\otimes \bar{\mathbb{F}}_p$ according to the Newton slopes of the isocrystal, and consider the rigid analytic tube over the unique closed stratum, the so-called basic locus $\overline{\mathscr{S}}^b$.
This tube $Sh^b$ is an open rigid analytic subvariety of $Sh_{\breve{E}_v}^{rig}$ (where $\breve{E}_v$ denotes the completion of the maximal unramified extension of $E_v$), and there is a uniformization
\begin{equation} \label{uniformization} Sh^b \cong I_{\phi}(\mathbb{Q}) \backslash \mathcal{M} \times G(\mathbb{A}_f^p)/K^p. \end{equation}
Here $\mathcal{M}$ is a so-called local Shimura variety (cf. \cite{Berkeley}) with an action of a $p$-adic Lie group $J_b(\mathbb{Q}_p)$, and $I_{\phi}(\mathbb{Q}) \subset J_b(\mathbb{Q}_p) \times G(\mathbb{A}_f^p)$ is a discrete subgroup.
One can write the right hand side of (\ref{uniformization}) as a finite disjoint union $$\coprod_g \Gamma_g \backslash \mathcal{M},$$ where $g$ runs through a set of representatives of the double coset $ I_{\phi}(\mathbb{Q}) \backslash G(\mathbb{A}_f^p) / K^p$.
For Shimura varieties of PEL type, this uniformization was introduced by Rapoport-Zink in the book \cite{RZ}, and later generalized to Shimura varieties of Hodge type (respectively abelian type) by Wansu Kim (respectively Xu Shen), cf. \cite{Wansu} and \cite{Xu}.
The local Shimura variety $\mathcal{M}$ can be defined as a moduli space of modifications of $G$-torsors on the Fargues-Fontaine curve (\cite{Berkeley}, Proposition 23.3.1), and thus the uniformization gives an interpretation of the rigid open subvariety $Sh^b$ of the $p$-adic moduli space in terms of group theory.

\subsection*{Moduli spaces of K3 surfaces}
\addcontentsline{toc}{subsection}{Moduli spaces of K3 surfaces}

In this paper, we study a Rapoport-Zink uniformization result for the moduli space of primitively polarized K3 surfaces, giving a similar group-theoretic description involving an explicit open subvariety of a local Shimura variety of orthogonal type.
The motive of a K3 surface is of abelian type, i.e. it lies in the subcategory of André motives generated by abelian varieties and Artin motives (cf. \cite{Andremotifs}, Théorème 7.1).
Therefore, even though the moduli space of polarized K3 surfaces is not itself a Shimura variety, it is still very closely related to a Shimura variety of abelian type, as we will now recall.

Let $n = q^r$ be a power of a prime $q \not=p$. For $n$ large enough and $2d$ not divisible by $p$, the moduli space of primitively polarized K3 surfaces of degree $2d$ with a level-$n$ structure (cf. Definition \ref{deflevelstructure}) will be representable by a smooth scheme $M_{2d, n, \mathbb{Z}_{(p)}} $ over $\mathbb{Z}_{(p)}$.

Denote by $\Lambda:= E_8(-1)^{\oplus 2} \oplus U^{\oplus 3}$ the K3 lattice, and $\lambda \in \Lambda$ an element with $\lambda^2 =2d$.
We will write $\Lambda_d := \langle \lambda \rangle^{\perp}$ for the complement of $\lambda$.
The analytic subspace $$\mathcal{D} = \{x \in \mathbb{P}(\Lambda_d \otimes \mathbb{C}) | \langle x,x\rangle = 0, \langle x,\overline{x} \rangle >0 \} \subset \mathbb{P}(\Lambda_d \otimes \mathbb{C})$$ is a Hermitian symmetric domain that can be viewed as the moduli space of Hodge structures of K3 type.
We define the arithmetic group $\mathcal{G}_n(\mathbb{Z})$ to be the subgroup of
$ \mathcal{G}(\mathbb{Z}) := \{ g \in SO(\Lambda) \, | \, g \lambda = \lambda\} $ of elements that reduce to the identity modulo $n$.
For $n \ge 3$, complex Hodge theory provides us with a period map $$\Phi: M_{2d, n, \mathbb{C}} \to \mathcal{G}_n(\mathbb{Z})\backslash \mathcal{D}$$
mapping a K3 surface to the Hodge structure on its second cohomology.
The right hand side is in fact a Shimura variety $Sh_{n,\mathbb{C}}$ attached to the reductive group $G := SO(\Lambda_{d, \mathbb{Q}})$ of signature $(2,19)$.
By the Torelli theorem for K3 surfaces (cf. \cite{Torelli}), the period map $\Phi$ is an open immersion that identifies $M_{2d,n, \mathbb{C}}$ with an algebraic open subvariety of $Sh_{n, \mathbb{C}}$.
One can explicitly determine the image of $\Phi$ as follows.
Denote by $\Delta(\Lambda_d) := \{\delta \in \Lambda_d \, | \, \delta^2=-2 \}$ the set of roots of the lattice $\Lambda_d$.
For $\delta \in \Delta(\Lambda_d)$ we can define the hypersurface $\delta^{\perp} \subset \mathbb{P}(\Lambda_d \otimes \mathbb{C})$.
The open analytic subspace $$\mathcal{D}^{\circ} = \mathcal{D} \setminus \bigcup_{\delta \in \Delta(\Lambda_d)} \delta^{\perp}$$ is the complement in $\mathcal{D}$ of what is called an arithmetic arrangement of hyperplanes.
Then the period map $\Phi$ induces a uniformization
$$M_{2d,n, \mathbb{C}} \cong \mathcal{G}_n(\mathbb{Z}) \backslash \mathcal{D}^{\circ}.$$
It is important to remark that in contrast to the Hermitian symmetric domain $\mathcal{D}$, the open subspace $\mathcal{D}^{\circ}$ is \emph{not} stable under the action of the real Lie group $G(\mathbb{R})$, but only under the discrete subgroup $\mathcal{G}(\mathbb{Z})$.

The main purpose of this paper is to establish a similar explicit $p$-adic uniformization for the moduli space of primitively polarized K3 surfaces with supersingular reduction.
The Shimura variety $Sh_{n,\mathbb{C}} = \mathcal{G}_n(\mathbb{Z})\backslash\mathcal{D}$ is naturally defined over the reflex field $E= \mathbb{Q}$.
It was first observed by Rizov (\cite{RizovCrelle}, Theorem 3.9.1) that the period map $\Phi$ descents to an open immersion of algebraic varieties $$\Phi: M_{2d,n, \mathbb{Q}} \hookrightarrow Sh_{n, \mathbb{Q}} $$ over $\mathbb{Q}$.
Provided that $p \nmid 2d$ and $n=q^r$ for a prime $q \not=p$, as we will assume throughout the paper, the Shimura variety $Sh_{n,\mathbb{Q}}$ has an integral canonical model $\mathscr{S}_{n}$ over $\mathbb{Z}_{(p)}$ as defined in \cite{Intmod},
and the period map extends to an open immersion
$$\varphi: M_{2d,n, \mathbb{Z}_{(p)}} \hookrightarrow \mathscr{S}_{n}$$ of schemes over $\mathbb{Z}_{(p)}$.
This integral extension of the period map was already used crucially by Madapusi-Pera in his proof of the Tate conjecture for K3 surfaces in odd characteristic \cite{MP}.

Together with the uniformization (\ref{uniformization}) for the tube over the basic locus of $\mathscr{S}_n$, it gives rise to a period map of rigid analytic spaces
\begin{equation} \label{padicper} \varphi: M_{2d,n, \breve{\mathbb{Q}}_{p}}^{ss} \hookrightarrow I_{\phi}(\mathbb{Q}) \backslash \mathcal{M} \times G(\mathbb{A}_f^p)/K_n^p = \coprod_g \Gamma_g \backslash \mathcal{M} \end{equation}
for the rigid analytic tube $M_{2d,n, \breve{\mathbb{Q}}_{p}}^{ss}$ over the supersingular locus $\overline{M}_{2d,ss,n} \subset M_{2d,n, \bar{\mathbb{F}}_{p}}$.

Here $\mathcal{M}$ is a local Shimura variety attached to the reductive group $G_{\mathbb{Q}_p} = SO(\Lambda_{d, \mathbb{Q}_p})$, and $I_{\phi}$ is the group of "quasi-isogenies" of a polarized supersingular K3 surface $(\mathbb{X}, \mathbb{L})$ over $\bar{\mathbb{F}}_p$, i.e. the special orthogonal group $SO(N_{d,\mathbb{Q}})$ attached to $N_{d, \mathbb{Q}} = \langle \mathbb{L}\rangle^{\perp} \subset NS(\mathbb{X})_{\mathbb{Q}}$.
For a detailed discussion of quasi-isogenies of K3 surfaces over $\bar{\mathbb{F}}_p$ we refer the reader to \cite{Isogenies}.
In order to obtain a Rapoport-Zink uniformization for $M_{2d,n, \breve{\mathbb{Q}}_{p}}^{ss}$, we compute the image of the period map (\ref{padicper}).
Recall that there is a $J_b(\mathbb{Q}_p)$-equivariant continuous specialization map $$sp: \left| \mathcal{M} \right| \to \left| X \right|$$ to an affine Deligne-Lusztig variety $X$ associated with $G_{\mathbb{Q}_p}$.
On the reduction, we have a period map of perfect schemes over $\bar{\mathbb{F}}_p$
\begin{equation}\label{redperiodmap}
\overline{\varphi}: \overline{M}_{2d,ss,n}^{perf} \hookrightarrow I_{\phi}(\mathbb{Q}) \backslash X \times G(\mathbb{A}_f^p)/K_n^p = \coprod_g \Gamma_g \backslash X. \end{equation}

For $g$ a representative of an element of $I_{\phi}(\mathbb{Q}) \backslash G(\mathbb{A}_f^p) / K^p$, we define a set of roots
$$ \Delta_g \subset \Delta(N_{d, \mathbb{Q}}) := \{\delta \in N_{d, \mathbb{Q}} \,| \, \delta^2 =-2\}$$
such that the natural action of $I_{\phi}(\mathbb{Q})$ on $\Delta(N_{d, \mathbb{Q}})$ restricts to an action of $\Gamma_g$ on $\Delta_g$.
For $\delta \in \Delta(N_{d, \mathbb{Q}})$ we introduce in Definition \ref{defZdelta}
a subset
$Z_{\delta} \subset X(\bar{\mathbb{F}}_p)$ and let $$X^{\circ}_g(\bar{\mathbb{F}}_p) := X(\bar{\mathbb{F}}_p) \setminus \bigcup_{\delta \in \Delta_g} Z_{\delta},$$
which is the set of $\bar{\mathbb{F}}_p$-points of an open subscheme $X_g^{\circ}\subset X$.
The action of $\Gamma_g$ on $X$ permutes the $Z_{\delta}$, and thus $\Gamma_g$ acts on $X^{\circ}_g$.

\begin{thm}[cf. Corollary \ref{unifordisj}]\label{introunifordisj}
Assume $p > 18d+4$.
The period map (\ref{redperiodmap}) induces a uniformization $$\overline{M}_{2d,n,ss}^{perf} \cong \coprod_g \Gamma_g \backslash X^{\circ}_g.$$
\end{thm}
Note that Theorem \ref{introunifordisj} gives a purely group-theoretic description of the underlying topological space of the supersingular locus $\overline{M}_{2d,n,ss}$.

Setting $\mathcal{M}_g^{\circ}:= sp^{-1}(X^{\circ}_g)$, an open rigid analytic subspace of $\mathcal{M}$, we obtain the desired Rapoport-Zink uniformization:

\begin{thm}[cf. Theorem \ref{padicuniformization}]\label{intropadicuniformization}
Assume $p > 18d+4$.
The moduli space of primitively polarized K3 surfaces of degree $2d$ with supersingular reduction admits a Rapoport-Zink type uniformization
$$ M_{2d,n, \breve{\mathbb{Q}}_{p}}^{ss} \cong \coprod_g \Gamma_g \backslash \mathcal{M}^{\circ}_g.$$
\end{thm}

\begin{rem}
\begin{enumerate}[(i)]
\item
Similarly to the complex case, the open subspace $\mathcal{M}^{\circ}_g$ is not stable under the action of the $p$-adic Lie group $J_b(\mathbb{Q}_p)$, but only under the discrete subgroup $\Gamma_g \subset J_b(\mathbb{Q}_p)$.
\item
As opposed to the Rapoport-Zink uniformization in the case of Shimura varieties of abelian type, in the disjoint union $ \coprod_g \Gamma_g \backslash \mathcal{M}^{\circ}_g$ not only the group $\Gamma_g$, but also the uniformizing space $\mathcal{M}^{\circ}_g$ depends on $g$.
\end{enumerate}
\end{rem}

\subsection*{Moduli spaces of cubic fourfolds}
\addcontentsline{toc}{subsection}{Moduli spaces of cubic fourfolds}

We briefly sketch how the same methods can be applied to obtain a Rapoport-Zink uniformization of the moduli space of smooth cubic fourfolds with supersingular reduction.

In this case, the lattice $\Lambda$ is the cohomology lattice corresponding to the $H^4$ of a cubic fourfold, and $\Lambda_0 := \langle \lambda \rangle^{\perp} \subset \Lambda$ the orthogonal complement of the square of the class of a hyperplane section.
The complex period map
$$\Phi: M_{cf,n, \mathbb{C}} \to \mathcal{G}_n(\mathbb{Z})\backslash \mathcal{D}$$
goes to a quotient of a homogeneous space $\mathcal{D}$ for the group $G(\mathbb{R})$ of signature $(2,20)$, where $G:= SO(\Lambda_{0, \mathbb{Q}})$.
Voisin showed in \cite{VoisinTorelli} that this period map is an open immersion.
As in (\cite{Laza}, Definition 2.16) one defines a set of roots $\Delta(\Lambda_0)$ and a set of long roots $\Delta_{long}(\Lambda_0)$ of $\Lambda_0$.

We set
$$\mathcal{D}_6 = \bigcup_{\delta \in \Delta(\Lambda_0)} \delta^{\perp} \subset \mathcal{D}$$
and
$$ \mathcal{D}_2 = \bigcup_{\delta \in \Delta_{long}(\Lambda_0)} \delta^{\perp} \subset \mathcal{D},$$
and define $\mathcal{D}^{\circ} := \mathcal{D} \setminus (\mathcal{D}_6 \cup \mathcal{D}_2)$.

Then as conjectured by Hassett in (\cite{Hassett}, §4.3), Laza and Loojenga in \cite{Laza} and \cite{Looijenga} computed that the image of the period map is precisely $\mathcal{D}^{\circ}$, giving an isomorphism
$$M_{cf, n, \mathbb{C}} \overset{\sim}{\to} \mathcal{G}_n(\mathbb{Z}) \backslash \mathcal{D}^{\circ}.$$

As in the K3 case, there is a $p$-adic period map for the tube over the supersingular locus
$$ M_{cf, n, \breve{\mathbb{Q}}_p}^{ss} \hookrightarrow I_{\phi}(\mathbb{Q})\backslash \mathcal{M} \times G(\mathbb{A}_f^p)/K_n^p= \coprod_g \Gamma_g \backslash \mathcal{M},$$
where $\mathcal{M}$ is a local Shimura variety attached to the special orthogonal group $G_{\mathbb{Q}_p} = SO(\Lambda_{0, \mathbb{Q}_p})$.
Here $I_{\phi}$ is the group $I_{\phi} = SO(CH^2_{0,\mathbb{Q}})$, where for a supersingular cubic fourfold $\mathbb{X}$ over $\bar{\mathbb{F}}_p$ we denote by $CH^2_{0, \mathbb{Q}}:= \langle h^2\rangle^{\perp} \subset CH^2_{\mathbb{Q}}:=CH^2(\mathbb{X})_{\mathbb{Q}}$ the complement of the square of a hyperplane section in the (rational) Chow group of codimension two.

For $g \in G(\mathbb{A}_f^p)$, we introduce the $\mathbb{Z}[\frac{1}{p}]$-lattice $CH^2_{0,g} := CH^2_{0,\mathbb{Q}} \cap g\Lambda_{0,\hat{\mathbb{Z}}^p}$, which depends only on the class of $g$ in the double coset $I_{\phi}(\mathbb{Q}) \backslash G(\mathbb{A}_f^p) / K_n^p$.
We define a set of roots $\Delta(CH^2_{0,g})$ and a set of long roots $\Delta_{long}(CH^2_{0,g})$.

As above, for $\delta \in CH^2_{0,g}$ there is the definition of a subset $Z_{\delta} \subset X(\bar{\mathbb{F}}_p)$.
We define
\begin{eqnarray*} Z_{6,g} := \bigcup_{\delta \in \Delta(CH^2_{0,g})} Z_{\delta}
& \textnormal{and} & Z_{2,g} := \bigcup_{\delta \in \Delta_{long}(CH^2_{0,g})} Z_{\delta} \end{eqnarray*}
and the open subset
$$ X^{\circ}_g(\bar{\mathbb{F}}_p) := X(\bar{\mathbb{F}}_p) \setminus ( Z_{6,g} \cup Z_{2,g}). $$
Denote by $\mathcal{M}^{\circ}_g := sp^{-1}(X^{\circ}_g)$ the open rigid analytic subvariety of the local Shimura variety $\mathcal{M}$ defined as the preimage of $X^{\circ}_g$ under the specialization map.

\begin{thm}[cf. Theorem \ref{cfuniformization}]\label{introcfuniformization}
For almost all $p$, the moduli space of smooth cubic fourfolds with supersingular reduction admits a uniformization
$$M_{cf, n, \breve{\mathbb{Q}}_p}^{ss} \cong \coprod_g \Gamma_g \backslash \mathcal{M}^{\circ}_g$$
of rigid spaces over $\breve{\mathbb{Q}}_p$, as well as of perfect schemes
$$\overline{M}_{cf,ss,n}^{perf} \cong \coprod_g \Gamma_g \backslash X^{\circ}_g $$
over $\bar{\mathbb{F}}_p$.
\end{thm}

\begin{ac}
The author would like to thank Laurent Fargues and Bruno Klingler for many helpful discussions and suggestions, and Sebastian Bartling for his comments on this text.
\end{ac}

\section{Moduli spaces of K3 surfaces}

We survey basic facts on K3 surfaces and their cohomology in characteristic $0$ and $p$. In addition, we recall how to define fine moduli spaces of polarized K3 surfaces with level structure. Relevant references are \cite{K3global}, \cite{MP} and \cite{Liedtke}.

\subsection{K3 surfaces and their cohomology}\label{sectioncoho}

Let $k$ be a field.

\begin{mydef}\label{K3def}
A \emph{K3 surface} over $k$ is a smooth projective surface $X$ over $k$ such that $\Omega_X^2 \cong \mathcal{O}_X $ and $H^1(X, \mathcal{O}_X)=0$.
A \emph{polarization} (resp. \emph{quasi-polarization}) of $X$ is a line bundle $\xi \in \mathrm{Pic}(X)$ which is ample (resp. big and nef).
The polarization (resp. quasi-polarization) $\xi$ is called \emph{primitive} if it is not a non-trivial multiple of some
element in $\mathrm{Pic(X)}$.
\end{mydef}

In this paper, we will assume all (quasi-)polarizations to be primitive.

The even, unimodular lattice $\Lambda:= E_8(-1)^{\oplus 2} \oplus U^{\oplus 3}$ of signature $(3,19)$ is called the K3 lattice.
Given a complex algebraic K3 surface $X$, there exists an isometry between the second Betti cohomology group $H_B^2(X, \mathbb{Z})$ together with its intersection form and the K3 lattice $\Lambda$ (cf. \cite{K3global}, Proposition 1.3.5).
Similarly, for a K3 surface $X$ over an arbitrary algebraically closed field of characteristic different from $\ell$, the étale cohomology groups $H^2_{\acute{e}t}(X, \mathbb{Z}_{\ell})$ are free of rank 22, and together with their intersection form are isometric to the lattice $\Lambda_{\mathbb{Z}_{\ell}}:=\Lambda \otimes \mathbb{Z}_{\ell}$.
If $X$ is a K3 surface over a perfect field $k$ of positive characteristic, then the crystalline cohomology $H^2_{cris}(X/W(k))$ is free of rank 22 and isometric to $\Lambda_{W(k)}:= \Lambda \otimes W(k)$.

Let $\lambda := e_1 + d f_1$ where $(e_1, f_1)$ denotes a basis of the first copy of the hyperbolic lattice $U$ in $\Lambda$, so that $\lambda^2=2d$. For $\xi$ a polarization of a complex K3 surface $X$ of degree $2d$, there exists an isometry between the lattices $H_B^2(X, \mathbb{Z})$ and $\Lambda$ mapping the Chern class $c_B(\xi) \in H_B^2(X, \mathbb{Z})$ to $\lambda$ (cf. \cite{K3global}, Corollary 14.1.10), and similarly for étale and crystalline cohomology.
We denote by $\Lambda_d := \langle\lambda\rangle^{\perp}\subset \Lambda$ the orthogonal complement of $\lambda$.
Hence $\Lambda_d \cong P_{B}^2(X, \mathbb{Z})$, where we use the notation $P_B^2(X,\mathbb{Z}) := \langle c_B(\xi) \rangle^{\perp} \subset H^2_B(X, \mathbb{Z})$ for the primitive cohomology, and similarly for $P_{\acute{e}t}^2(X, \mathbb{Z}_{\ell})$ and $P^2_{cris}(X/W(k))$.

Define $G$ to be the reductive group over $\mathbb{Q}$ given by $SO(\Lambda_{d, \mathbb{Q}})$. Then $G_{\mathbb{R}}$ can be identified with the special orthogonal group $SO_{2,19}$.
A natural integral model for $G$ is the group scheme $\mathcal{G}$ over $\mathbb{Z}$ with generic fiber $\mathcal{G}_{\mathbb{Q}}=G$ whose $R$-points are given by
\begin{equation}\label{integralgroup} \mathcal{G}(R) = \{ g \in SO(\Lambda \otimes R) \, | \, g \lambda = \lambda\}. \end{equation}
It is important to remark that this is not just the $\mathbb{Z}$-group scheme $SO(\Lambda_d)$, the reason being that it is not always possible to extend a self-isometry of $\Lambda_d$ to one of $\Lambda$ which fixes $\lambda$.
We will also need the corresponding orthogonal group $O:= O(\Lambda_{d, \mathbb{Q}})$ with integral model
\begin{equation}\label{orthgroup} \mathcal{O}(R) = \{ g \in O(\Lambda \otimes R) \, | \, g \lambda = \lambda\}. \end{equation}

\subsection{Moduli spaces}\label{modspace}

In order to define moduli spaces, we need the following relative version of K3 surfaces over a general base scheme.

\begin{mydef}[\cite{MP}, 2.1] We define a \emph{K3 surface over a scheme} $S$ to be a smooth proper morphism of algebraic spaces $f: X \to S$ such that the fiber over any geometric point $\overline{s} \to S$ is a K3 surface in the sense of Definition \ref{K3def}.
A \emph{primitive polarization} (resp. \emph{primitive quasi-polarization}) of $f: X \to S$ is a global section $\xi \in \mathrm{Pic}_{X/S}(S)$ such that for every geometric point $\overline{s} \to S$ the fiber $\xi_{\overline{s}}$ is a primitive polarization (resp. primitive quasi-polarization).
\end{mydef}

\begin{rem}
We check the condition to be a primitive polarization on geometric points, since a polarization which is primitive over an arbitrary field could fail to be primitive over a finite extension.
\end{rem}

Denote by $\mathfrak{M}_{2d}$ the moduli problem that attaches to each $\mathbb{Z}[\frac{1}{2}]$-scheme $S$ the groupoid of pairs $(f: X \to S, \xi)$ consisting of a K3 surface $f: X\to S$ and a primitive polarization $\xi$ of degree $2d$.
Similarly, we denote by $\mathfrak{M}_{2d}^{*}$ the moduli problem attaching to each $\mathbb{Z}[\frac{1}{2}]$-scheme $S$ the groupoid of K3 surfaces over $S$ together with a primitive quasi-polarization of degree $2d$.

\begin{prop}[\cite{Rizovmoduli}; \cite{MP}, Proposition 2.2]
Both $\mathfrak{M}_{2d}$ and $\mathfrak{M}^{*}_{2d}$ are Deligne-Mumford stacks of finite type over $\mathbb{Z}[\frac{1}{2}]$, and $\mathfrak{M}_{2d}$ is separated.
The natural map $\mathfrak{M}_{2d} \to \mathfrak{M}^{*}_{2d}$ is an open immersion and the complement is of pure (relative) codimension one.
\end{prop}

\begin{rem}
The stack $\mathfrak{M}^{*}_{2d}$ is not separated, cf. (\cite{K3global}, 5.1.4).
\end{rem}

For every prime $\ell$, the relative étale cohomology of the universal object $f:\mathfrak{X} \to \mathfrak{M}_{2d}$ gives rise to an $\ell$-adic étale local system
$$ \mathcal{H}_{\ell} := R^2f_* \underline{\mathbb{Z}}_{\ell, \mathfrak{X}}(1)$$
of rank $22$ over $\mathfrak{M}_{2d, \mathbb{Z}[\frac{1}{2\ell}]}$, where $(1)$ denotes a Tate twist.
Furthermore, there is a perfect Poincaré pairing
$$ \langle -, - \rangle: \mathcal{H}_{\ell} \times \mathcal{H}_{\ell} \to \underline{\mathbb{Z}}_{\ell}.$$
The Chern class of the universal polarization $\mathfrak{L}$ gives rise to a global section $c_{\ell}(\mathfrak{L})$ of $\mathcal{H}_{\ell}$ which satisfies
$$\langle c_{\ell}(\mathfrak{L}) , c_{\ell}(\mathfrak{L})\rangle = 2d.$$

Choose a prime number $p \ge 3$.
We can form the $\hat{\mathbb{Z}}^p$-local system $\mathcal{H}_{\hat{\mathbb{Z}}^p} := \prod_{\ell \not=p} \mathcal{H}_{\ell}$ over $\mathfrak{M}_{2d, \mathbb{Z}_{(p)}}$. The Chern classes $c_{\ell}(\mathfrak{L})$ assemble into a Chern class $c_{\hat{\mathbb{Z}}^p}(\mathfrak{L})$ in
$\mathcal{H}_{\hat{\mathbb{Z}}^p}$.

We let $\mathcal{I}^p$ be the étale sheaf over $\mathfrak{M}_{2d, \mathbb{Z}_{(p)}}$
whose sections over any scheme $T\to \mathfrak{M}_{2d, \mathbb{Z}_{(p)}}$ are given by
$$\mathcal{I}^p(T) = \left\{\eta: \Lambda\otimes \underline{\hat{\mathbb{Z}}}^p_{T} \overset{\sim}{\rightarrow} \mathcal{H}_{\hat{\mathbb{Z}}^p, T} \textnormal{\, such that \,} \eta(\lambda)= c_{\hat{\mathbb{Z}}^p}(\mathfrak{L})\right\},$$ i.e. the sheaf of isometries between $\mathcal{H}_{\hat{\mathbb{Z}}^p, T}$ and the constant sheaf attached to the K3 lattice $\Lambda_{\hat{\mathbb{Z}}^p}$.
The sheaf $\mathcal{I}^p$ has a natural right action by the constant sheaf of groups
$\mathcal{O}(\hat{\mathbb{Z}}^p)$ as defined in (\ref{orthgroup}) in Section \ref{sectioncoho}.
For a number $n$ coprime to $p$ we define the congruence subgroup $$\mathcal{O}_n(\hat{\mathbb{Z}}^p) := \left\{g \in \mathcal{O}(\hat{\mathbb{Z}}^p) |\, g \equiv 1 \mod n \right\}.$$
A section $\eta \in H^0(T, \mathcal{I}^p/\mathcal{O}_n(\hat{\mathbb{Z}}^p))$ is called a \emph{level-$n$ structure} over $T$.

\begin{mydef}[\cite{MP}, 2.10]\label{deflevelstructure} We define the stack $\mathfrak{M}_{2d, n, \mathbb{Z}_{(p)}}$ as the relative moduli problem over $\mathfrak{M}_{2d, \mathbb{Z}_{(p)}}$ that attaches to $T \to \mathfrak{M}_{2d, \mathbb{Z}_{(p)}} $ the set of level-$n$ structures over $T$.
\end{mydef}

\begin{prop}[\cite{Rizovmoduli}; \cite{MP}, Proposition 2.11]
The stack $\mathfrak{M}_{2d, n, \mathbb{Z}_{(p)}}$ is finite étale over $\mathfrak{M}_{2d, \mathbb{Z}_{(p)}}$. For $n$ large enough, it admits a fine moduli space $M_{2d, n, \mathbb{Z}_{(p)}}$ which is a scheme over $\mathbb{Z}_{(p)}$.
Moreover, it is smooth if $p \nmid 2d$.
\end{prop}
Here the fact that $M_{2d,n, \mathbb{Z}_{(p)}}$ is representable by a scheme (and not merely an algebraic space) follows a posteriori from the existence of an open immersion into an integral canonical model of a Shimura variety (\cite{MP}, Corollary 4.15).

\begin{rem}\label{remarkorientation}
In \cite{MP}, compare also (\cite{Taelman}, 5.1), the author defines a 2:1 étale cover of $\mathfrak{M}_{2d,\mathbb{Z}_{(p)}} $ by choosing an \emph{orientation}, i.e. a trivialization of the determinant $$\mathrm{det}(\Lambda\otimes \underline{\hat{\mathbb{Z}}}^p) \cong \mathrm{det}(\mathcal{H}_{\hat{\mathbb{Z}}^p})$$ which over $\mathfrak{M}_{2d,\mathbb{C}}$ is induced by a trivialization
$ \mathrm{det}(\underline{\Lambda}) \cong \mathrm{det}(\mathcal{H}_{B})$ of the determinant of Betti cohomology.
One should then define \emph{oriented level structures} which respect the choice of the orientation (\cite{Taelman}, Definition 5.11).
\emph{In this paper, we will assume from now on that $n = q^r \ge 3$ is a power of a prime number $q \not= p$}.
In this special case, the trivialization $ \mathrm{det}(\Lambda\otimes \underline{\mathbb{Z}/n\mathbb{Z}}) \cong \mathrm{det}(\mathcal{H}_{\mathbb{Z}/n\mathbb{Z}})$ given by the universal level-$n$ structure over $M_{2d, n, \mathbb{Z}_{(p)}}$ as above lifts uniquely to an orientation $\mathrm{det}(\Lambda\otimes \underline{\hat{\mathbb{Z}}}^p) \cong \mathrm{det}(\mathcal{H}_{\hat{\mathbb{Z}}^p})$
over $M_{2d, n, \mathbb{Z}_{(p)}}$.
Restricting to the case where $n =q^r$ therefore removes the necessity of choosing an orientation, and resolves the issue raised in (\cite{Taelman}, Remark 5.12). \end{rem}

Similarly, one defines finite étale covers $\mathfrak{M}_{2d, n, \mathbb{Z}_{(p)}}^{*}$ of the stack $\mathfrak{M}_{2d, \mathbb{Z}_{(p)}}^{*}$ of quasi-polarized K3 surfaces, which are representable by algebraic spaces $M_{2d, n, \mathbb{Z}_{(p)}}^{*}$ for $n$ sufficiently large (\cite{MP}, Proposition 2.11).

For $n$ large enough, we have the usual cohomological realizations of the universal polarized K3 surface $f:\mathfrak{X} \to M_{2d,n, \mathbb{Z}_{(p)}}$.
The relative Betti cohomology gives rise to a pure $\mathbb{Z}$-variation of Hodge structure
$$ \mathcal{H}_{B} := R^2f^{an}_{\mathbb{C},*} \underline{\mathbb{Z}}_{\mathfrak{X}_{\mathbb{C}}^{an}}(1)$$
of weight zero on $M_{2d,n, \mathbb{C}}$.
We denote by $(\mathcal{H}_{dR}, \nabla)$ the associated algebraic vector bundle with flat connection, which is actually defined over $M_{2d,n,\mathbb{Z}_{(p)}}$.
We use the notation $\mathcal{P}_B := \langle c_B(\mathfrak{L})\rangle^{\perp} \subset \mathcal{H}_B$ for the sub-variation corresponding to primitive cohomology.
Similarly, for $\ell \not= p$ we write $\mathcal{P}_{\ell}:= \langle c_{\ell}(\mathfrak{L})\rangle^{\perp} \subset \mathcal{H}_{\ell}$ for the primitive part of the $\ell$-adic étale local system over $M_{2d, n,\mathbb{Z}_{(p)}}$. The restricted Poincaré pairing
$$ \langle -, - \rangle: \mathcal{P}_{\ell} \times \mathcal{P}_{\ell} \to \underline{\mathbb{Z}}_{\ell}$$
is perfect for $\ell \nmid d$.

\subsection{K3 surfaces in positive characteristic}\label{K3pos}

For $X$ a smooth projective variety over $\bar{\mathbb{F}}_p$, crystalline cohomology has the additional feature of a Frobenius operator coming from the absolute Frobenius on $X$, equipping $H^k_{cris}(X/W)$ with the structure of an $F$-crystal.
Here we denote by $W:= W(\bar{\mathbb{F}}_p)$ the ring of Witt vectors of $\bar{\mathbb{F}}_p$, with its Frobenius $\sigma$ and field of fractions $\breve{\mathbb{Q}}_p$. We will now recall basic facts about $F$-crystals and study those arising from the second cohomology of a K3 surface. References are \cite{Liedtke} and \cite{Kottwitz}, \cite{Kottwitz2}.

\begin{mydef}[\cite{Liedtke}, Definition 3.1]
An ($F$-)crystal is a pair $(M, \varphi_M)$ consisting of a finite free $W$-module $M$ together with an injective $\sigma$-linear endomorphism
$\varphi_M: M \to M$.
An ($F$-)isocrystal is a pair $(V, \varphi_V)$ where $V$ is a finite dimensional $\breve{\mathbb{Q}}_p$-vector space together with a $\sigma$-linear automorphism $\varphi_V: V \to V$.
\end{mydef}

Let $r, s$ be two coprime integers with $s >0$ and $\alpha:= r/s \in \mathbb{Q}$. We can define an $F$-isocrystal $(V_{\alpha}, \varphi_{\alpha})$ in the following way.
Set $V_{\alpha}:= \breve{\mathbb{Q}}_p^{s}$ and let $\varphi_{\alpha} = b_{\alpha} \sigma$, where
$b_{\alpha}$ is given by the matrix
$$ \begin{pmatrix} 0 & 1 & & & \\
& 0 & 1 & & & \\
& & \ddots & \ddots & & \\
& & & 0 & 1 \\
p^{r} & & & & 0
\end{pmatrix}.$$

We have the following well-known classification result for $F$-isocrystals.
\begin{thm}[Dieudonné-Manin, cf. \cite{Manin}]
The category of $F$-isocrystals is semisimple, with simple objects isomorphic to $(V_{\alpha}, \varphi_{\alpha})$ for $\alpha \in \mathbb{Q}$.
\end{thm}
In particular, every $F$-isocrystal $(V, \varphi_V)$ decomposes as a direct sum $$(V, \varphi_V) \cong \bigoplus_{\alpha \in \mathbb{Q}} (V_{\alpha}, \varphi_{\alpha})^{n_{\alpha}}.$$
The rational numbers $\alpha \in \mathbb{Q}$ which show up in this decomposition (i.e. $n_{\alpha} \not= 0$) are called the \emph{Newton slopes} of $(V, \varphi_V)$, and the $n_{\alpha}$ the corresponding multiplicities.

For an $F$-isocrystal $(V, \varphi_V)$ of rank $n$, one can visualize the Newton slopes by drawing its \emph{Newton polygon}.
Ordering the Newton slopes in an increasing fashion $$\alpha_1 < \alpha_2 < \cdots,$$
the Newton polygon is the piecewise linear function $[0, n] \to \mathbb{R}$ with the following slopes:
$$\begin{matrix} \mathrm{slope \, } & \alpha_1 & \textnormal{if \,} 0 \le t < n_{\alpha_1}; \\
\mathrm{slope \, } & \alpha_2 & \textnormal{if \,} n_{\alpha_1} \le t < n_{\alpha_2}; \\
& \cdots &
\end{matrix}$$

There is a group-theoretic description of the set of isomorphism classes of $F$-isocrystals in terms of the Kottwitz set.
\begin{mydef}[\cite{Kottwitz}, 1.7]
For a reductive group $G$ over $\mathbb{Q}_p$, we define the \emph{Kottwitz set} $B(G)$ as the quotient
$$ B(G) := G(\breve{\mathbb{Q}}_p) / \sim,$$
where we introduce the equivalence relation $b \sim b'$ if $b' = gb \sigma(g)^{-1}$ for some $g \in G(\breve{\mathbb{Q}}_p)$.
\end{mydef}

For any element $b \in GL_n(\breve{\mathbb{Q}}_p)$, we can define an $F$-isocrystal $(V_b, \varphi_{b})$ by setting $V_b:=\breve{\mathbb{Q}}_p^n$ and
$\varphi_{b}:= b \sigma $.
One can check that the isocrystals constructed in this way are isomorphic exactly if $b \sim b'$. As a consequence, there is a bijection
between the set of isomorphism classes of $F$-isocrystals of rank $n$ and the elements in the Kottwitz set $B(GL_n)$.
More generally, we refer to the elements of $B(G)$ as \emph{$G$-isocrystals}.

The Kottwitz set $B(G)$ has two important invariants, the Newton map and the Kottwitz map.
The Newton map $\nu: B(G) \to \mathcal{N}(G)$ goes from the Kottwitz set to the Newton chamber
$$\mathcal{N}(G) := \left(\mathrm{Hom}(\mathbb{D}_{\bar{\mathbb{Q}}_p}, G_{\bar{\mathbb{Q}}_p}) / G(\bar{\mathbb{Q}}_p)- \mathrm{conjugation}\right)^{\mathrm{Gal}(\bar{\mathbb{Q}}_p / \mathbb{Q}_p)},$$
the set of $G(\bar{\mathbb{Q}}_p)$-conjugacy classes fixed by the absolute Galois group of $\mathbb{Q}_p$ of morphisms from the slope pro-torus $\mathbb{D}$ with character group $\mathbb{Q}$ to $G_{\bar{\mathbb{Q}}_p}$.
The set $\mathcal{N}(G) $ can be identified with the set $X^{*}(T)^+_{\mathbb{Q}}$ of rational positive cocharacters of the maximal unramified torus $T$ of a quasi-split inner form of $G$. Via this description, $\mathcal{N}(G)$ is equipped with an order $\ge$, where $v_2 \ge v_1$ if and only if $v_2 -v_1 \in \mathbb{Q}_{\ge 0} \Phi_0^+$ for $\Phi_0^+$ the set of positive roots of $T$.
For $G=GL_n$, the element $\nu_b \in \mathcal{N}(G)$ corresponds to the collection of Newton slopes of the isocrystal $(V_b, \varphi_{b})$ via the identification $X^{*}(T)_{\mathbb{Q}} = \mathbb{Q}^n$.
In general, the Newton map can be defined using the Tannakian formalism, cf. (\cite{Kottwitz}, 4.2).

The other map is the Kottwitz map $\kappa: B(G) \to \pi_1(G)_{\mathrm{Gal}(\bar{\mathbb{Q}}_p / \mathbb{Q}_p)}$.
We will not recall it here, but the interested reader can find a definition in (\cite{Kottwitz2}, 4.9 and 7.5).
The map $B(G) \overset{\nu \times \kappa}{\to} \mathcal{N}(G) \times \pi_1(G)_{\mathrm{Gal}(\bar{\mathbb{Q}}_p / \mathbb{Q}_p)}$ is injective, i.e. an element of $B(G)$ is uniquely characterized by its images under the Newton and Kottwitz maps (cf. \cite{Kottwitz2}, 4.13).

A conjugacy class of cocharacters $\mu: \mathbb{G}_{m, \bar{\mathbb{Q}}_p} \to G_{\bar{\mathbb{Q}}_p}$ induces an element $\mu^{\diamond} \in \mathcal{N}(G)$, as well as an element $\mu^{\sharp} \in \pi_1(G)_{\mathrm{Gal}(\bar{\mathbb{Q}}_p / \mathbb{Q}_p)}$.
For $G=GL_n$, such a conjugacy class is given by a collection of distinct integers $\mu_1 < … < \mu_i$ with multiplicities $m_{\mu_1}, …, m_{\mu_i}$.
The \emph{Hodge polygon} of $\mu$ is defined as the polygon with slopes $\mu_i$ and multiplicities $m_{\mu_i}$, similarly as above.

\begin{mydef}[\cite{Kottwitz2}, 6.2]
We define the \emph{admissible set} $$B(G, \mu):= \{ b \in B(G)| [\nu_b] \le \mu^{\diamond}, \kappa({b}) = \mu^{\sharp}\}.$$
\end{mydef}
\begin{rem} For $G=GL_n$, this is precisely the set of $b \in B(G)$ such that the Newton polygon of $(V_b, \varphi_b)$ lies on or above the Hodge polygon of $\mu$, and both have the same endpoint.
\end{rem}

We study the isocrystals arising from the second cohomology of a quasi-polarized K3 surface $(X, \xi)$ over $\bar{\mathbb{F}}_p$.
The crystalline cohomology $H_{cris}^2(X/W)$ is a free $W$-module of rank 22 equipped with a Frobenius $\varphi$ which becomes an isomorphism after inverting $p$. We consider the $F$-isocrystal
$$\mathcal{I}(X):= P^2_{cris}(X/W)[\frac{1}{p}](1) = \langle c_{cris}(\xi)\rangle^{\perp} \subset H^2_{cris}(X/W)[\frac{1}{p}](1),$$
where $(1)$ again denotes a Tate twist.
Via the choice of an isometry $P^2_{cris}(X/W)[\frac{1}{p}] \cong \Lambda_{d,\mathbb{Q}_p}$, the isocrystal $\mathcal{I}(X)$ corresponds to an element of $B(GL(\Lambda_{d,\breve{\mathbb{Q}}_p}))$. The fact that the intersection pairing
$$ \mathcal{I}(X) \otimes \mathcal{I}(X) \to \breve{\mathbb{Q}}_p$$
is a morphism of $F$-isocrystals and the determinant $\mathrm{det}(\mathcal{I}(X))$ is of slope $0$ shows that the isocrystal lies in the Kottwitz set $B(G)$ for $G = SO(\Lambda_{d,\mathbb{Q}_p})$.

Let $\mu: \mathbb{G}_{m,\mathbb{Q}_p} \to G$ be the cocharacter \begin{eqnarray*}\mu: \mathbb{G}_{m, \mathbb{Q}_p} & \to & G, \\
t & \mapsto & \left(\begin{array}{rrrrr}t & & &
& \\ & 1 & & & \\ & & \ddots & & \\ & & & 1 & \\ & & & & t^{-1} \end{array} \right) \end{eqnarray*}
recording the Hodge numbers $(1,19,1)$ of the primitive second cohomology of a quasi-polarized K3 surface.
A classical result of Berthelot-Ogus (\cite{BO}; see also the survey \cite{Liedtke}, Theorem 3.8) asserts that the Newton polygon of $\mathcal{I}(X)$ lies above or on the Hodge polygon of $\mu$, and both have the same endpoint.
This shows that the $G$-isocrystal attached to $\mathcal{I}(X)$ lies in fact in $B(G, \mu)$.

Conversely, one can show that in fact all elements of $B(G, \mu)$ arise from a quasi-polarized K3 surface over $\bar{\mathbb{F}}_p$, see \cite{CyclesK3}.

One can check that the possible Newton slopes of the isocrystals coming from elements in $B(G,\mu)$ are the ones given by
$$\underbrace{-\frac{1}{h}, \cdots, -\frac{1}{h}}_h, \underbrace{0, \cdots , 0}_{21-2h}, \underbrace{\frac{1}{h}, \cdots, \frac{1}{h}}_h $$ for $h \in \{1, \cdots, 10, \infty\}$, where the case $h = \infty$ corresponds to the isoclinic isocrystal of slope $0$.
We denote the corresponding elements of $B(G, \mu)$ by $b_h$ for $h \in \{1, \cdots, 10, \infty\} $.
\begin{mydef}
A K3 surface $X$ whose associated $G$-isocrystal is $b_h$ is called \emph{of height} $ h \in \{1, \cdots, 10, \infty\}$. A K3 surface of infinite height is called \emph{supersingular}.
\end{mydef}
Note that $\nu_{b_{\infty}} \le \nu_{b_{10}} \le \cdots \le \nu_{b_1}$ in terms of the order in $\mathcal{N}(G)$.
This induces a natural height stratification by closed subvarieties
$$ M_{2d,n, \bar{\mathbb{F}}_p} = \overline{M}_1 \supset \cdots \supset \overline{M}_{10} \supset \overline{M}_{2d,ss,n},$$
where for $1 \le h \le 10$, the subvariety $M_h$ is the subspace of K3 surfaces of height $\ge h$, and $\overline{M}_{2d,ss,n} $ is the supersingular locus.
The height stratification was first introduced by Artin in \cite{Artin}.
The supersingular locus is a $9$-dimensional closed algebraic subvariety (cf. \cite{Ogus}, Theorem 2.7).

From now on, we assume $p \ge 3$.
It follows from the integral Tate conjecture (\cite{MP}, Theorem 1) and (\cite{K3crystals}, Corollary 1.6) that the Neron-Severi group $NS(X)$
of a supersingular K3 surface $X$ has maximal rank $22$, and the Neron-Severi group generates the entire cohomology:
\begin{eqnarray}\label{generatescoho} NS(X) \otimes \mathbb{Z}_{\ell} & \cong & H^2_{\acute{e}t}(X, \mathbb{Z}_{\ell}); \\ \label{generatescoho2}
NS(X) \otimes \mathbb{Z}_p & \cong & H_{cris}^2(X/W)^{\varphi=p}. \end{eqnarray}
The statement in (\cite{MP}, Theorem 1) is just the rational version of the Tate conjecture, but it is known that for divisors, this version of the conjecture implies the integral Tate conjecture (cf. \cite{Tateconjecture}).

\begin{prop}\label{NSintegral}
For $X$ a supersingular K3 surface over $\bar{\mathbb{F}}_p$ let $$NS(X)_{\mathbb{Q}} \hookrightarrow \prod_{\ell \not= p} H^2_{\acute{e}t}(X, \mathbb{Q}_{\ell}) \times H^2_{cris}(X/W)[\frac{1}{p}]$$ denote the diagonal embedding.
Then $$NS(X) = NS(X)_{\mathbb{Q}} \cap \prod_{\ell \not= p} H^2_{\acute{e}t}(X, \mathbb{Z}_{\ell}) \times H^2_{cris}(X/W).$$
\end{prop}
\begin{proof}
It follows from (\ref{generatescoho}) and (\ref{generatescoho2})
that $$NS(X) \otimes \mathbb{Z}_{(\ell)} = NS(X)_{\mathbb{Q}} \cap H^2_{\acute{e}t}(X, \mathbb{Z}_{\ell})$$
for $\ell \not=p$ and $$ NS(X) \otimes \mathbb{Z}_{(p)} = NS(X)_{\mathbb{Q}} \cap H^2_{cris}(X/W).$$
This proves the proposition since $NS(X)$ is torsion-free.
\end{proof}

\section{Shimura varieties of orthogonal type}\label{sectionSh}

In this section we study the period map for the moduli space of polarized K3 surfaces towards a Shimura variety of orthogonal type. References are \cite{Torelli}, \cite{DeligneShimura}, \cite{RizovCrelle}.

\subsection{The complex period map}

The Hodge diamond of a complex K3 surface $X$ has the form
$$\begin{matrix} &&1&& \\ &0&&0& \\ 1& & 20 & & 1 \\ &0&&0& \\ &&1&&
\end{matrix}. $$
In particular, the Hodge structure $$H_B^2(X, \mathbb{C}) = H^{0,2}(X) \oplus H^{1,1}(X) \oplus H^{2,0}(X)$$ is completely determined by the one-dimensional subspace $H^{2,0}(X) \subset H_B^2(X, \mathbb{C}).$
Name\-ly, given $H^{2,0}(X)$, we can recover $H^{0,2}(X) = \overline{H^{2,0}(X)} $ and $H^{1,1}(X) = (H^{2,0}(X))^{\perp} $.
The complex analytic subspace $$\mathcal{D} = \{x \in \mathbb{P}(\Lambda_d \otimes \mathbb{C}) | \langle x,x \rangle=0, \langle x,\overline{x}\rangle >0 \} \subset \mathbb{P}(\Lambda_d \otimes \mathbb{C})$$
can thus be viewed as the complex period domain of Hodge structures of the type arising from polarized K3 surfaces of degree $2d$.
Recall that we defined the $\mathbb{Q}$-group $G:= SO(\Lambda_{d,\mathbb{Q}})$. The group $G(\mathbb{R})$ acts transitively on $\mathcal{D}$.
Therefore the manifold $\mathcal{D}$ can also be viewed as the $G(\mathbb{R})$-orbit of a Hodge cocharacter $h: \mathbb{S} \to G_{\mathbb{R}}$.

Defining the congruence subgroup
\begin{equation}\label{groupdef}\mathcal{O}_n(\mathbb{Z}) := \left\{g \in \mathcal{O}(\mathbb{Z}) |\, g \equiv 1 \mod n \right\}, \end{equation}
the monodromy representation $\rho: \pi_1(M_{2d,n,\mathbb{C}}^{an}) \to \mathcal{O}(\mathbb{Z})$ corresponding to the local system $\mathcal{H}_B$ on $M_{2d,n,\mathbb{C}}$ factors through $\mathcal{O}_n(\mathbb{Z})$.

A fundamental result in Hodge theory is the Torelli theorem for K3 surfaces due to Pjateckii-Shapiro and Shafarevic.

\begin{thm}[\cite{Torelli}]
The complex period map
\begin{equation}\label{pemap} \Phi: M_{2d,n,\mathbb{C}}^{an} \to \mathcal{O}_n(\mathbb{Z}) \backslash \mathcal{D} \end{equation} is an open immersion.
\end{thm}

One can show that without level structure the stack $\mathfrak{M}_{2d, \mathbb{C}}$ admits a coarse moduli space $M_{2d,\mathbb{C}}$ and a period map $$ \Phi: M_{2d,\mathbb{C}}^{an} \hookrightarrow \mathcal{O}(\mathbb{Z}) \backslash \mathcal{D}. $$
Let us recall the explicit description of this period map on complex points.
For a complex point $x \in M_{2d,\mathbb{C}}(\mathbb{C})$ corresponding to a polarized complex K3 surface $(X, \xi)$, we
choose a \emph{marking} of $(X, \xi)$, i.e. an isometry $\iota: \Lambda \overset{\sim}{\rightarrow} H_B^2(X, \mathbb{Z})$ mapping $\lambda $ to $c_B(\xi)$.
Via the marking $\iota$, the one-dimensional subspace $ H^{2,0}(X) \subset H_B^2(X, \mathbb{Q}) \otimes \mathbb{C}$ defines a one-dimensional subspace $\iota^{-1} (H^{2,0}(X)) \subset \Lambda_d \otimes \mathbb{C}$.
The image of $x$ in $\mathcal{D}$ with respect to the marking $\iota$ is defined to be the corresponding point in $\mathcal{D} \subset \mathbb{P}(\Lambda_{d,\mathbb{C}})$.
Changing the marking translates into a left action of $\mathcal{O}(\mathbb{Z})$ on $\mathcal{D}$, and therefore after taking the quotient by this action, this gives a well-defined point $\Phi(x) \in \mathcal{O}(\mathbb{Z}) \backslash \mathcal{D}$.

The injectivity of the period map thus translates into the following theorem:

\begin{thm}[\cite{Torelli}]
Let $(X,\xi)$ and $(X', \xi')$ be two polarized K3 surfaces over $\mathbb{C}$ and let $\iota: H_B^2(X',\mathbb{Z}) \overset{\sim}{\rightarrow} H_B^2(X,\mathbb{Z}) $ be an isometry mapping $c_B(\xi')$ to $c_B(\xi)$ and compatible with the Hodge structures.
Then there exists a unique isomorphism $f: X \overset{\sim}{\rightarrow} X'$ inducing $\iota$.
\end{thm}

\subsection{Shimura varieties of orthogonal type}\label{Shsection}

It is important to note that for $n \ge 3$ the target of the period map (\ref{pemap}) is in fact a Shimura variety. Indeed, the pair $(G, \mathcal{D})$ forms a Shimura datum in the sense of (\cite{DeligneShimura}, 1.5).
For a compact open subgroup $K \subset G(\mathbb{A}_f)$, the associated Shimura variety is defined as the double quotient \begin{equation}\label{ShSO} Sh_{K, \mathbb{C}}:= G(\mathbb{Q}) \backslash \mathcal{D} \times G(\mathbb{A}_f) / K. \end{equation}
We will mainly work with congruence subgroups, i.e. subgroups of the form
$$K_n := \ker(\mathcal{G}(\hat{\mathbb{Z}}) \to \mathcal{G}(\mathbb{Z}/n\mathbb{Z})),$$
and denote the corresponding Shimura variety by $Sh_{n, \mathbb{C}}:= Sh_{K_n, \mathbb{C}}.$

Using strong approximation for the Spin cover of $G$, one can show that
$$Sh_{n, \mathbb{C}} = \mathcal{G}_n(\mathbb{Z}) \backslash \mathcal{D},$$
observing that $\mathcal{G}_n(\mathbb{Z}) = K_n \cap G(\mathbb{Q})$.

As soon as $n \ge 3$, the congruence condition in (\ref{groupdef}) forces the determinant to be $1$, and hence the groups $\mathcal{O}_n(\mathbb{Z})= \mathcal{G}_n(\mathbb{Z})$ are equal.
In particular, for $n \ge 3$ the target of the period map $\Phi$ in the previous section is the Shimura variety $Sh_{n, \mathbb{C}}$.

By the general theory of Shimura varieties, $Sh_{n, \mathbb{C}}$ is a quasi-projective algebraic variety and has a canonical model $Sh_{n, \mathbb{Q}}$ defined over $\mathbb{Q}$.

A classical result of Borel (\cite{Borel}, Theorem 3.10) shows that the period map $\Phi: M_{2d,n,\mathbb{C}} \to Sh_{n, \mathbb{C}}$ is a morphism of complex algebraic varieties.

Assume now that $n = q^{r} \ge 3$ is the power of a prime number $q$ (cf. Remark \ref{remarkorientation}). The following result due to Rizov states that the period map in fact descents to a morphism over $\mathbb{Q}$.

\begin{thm}[\cite{RizovCrelle}, Theorem 3.9.1; \cite{MP}, Corollary 4.4]
The period map $\Phi$ descents to an open immersion
\begin{equation}\label{periodQ}
\Phi_{\mathbb{Q}}: M_{2d,n, \mathbb{Q}} \hookrightarrow Sh_{n,\mathbb{Q}}\end{equation}
of algebraic varieties over $\mathbb{Q}$.
\end{thm}

The proof in \cite{RizovCrelle} uses the fact that one can control the behaviour of the period map on the Zariski dense set of CM points.
A different approach (cf. \cite{MP}, Corollary 4.4) consists of showing that the period map is defined by a correspondence which is absolute Hodge.

The natural representation $\mathcal{G}_n(\mathbb{Z}) \to GL(\Lambda_d)$ defines a $\mathbb{Z}$-variation of Hodge structure $\mathbb{V}$ on $Sh_{n, \mathbb{C}}$, and the associated $\ell$-adic étale components $\mathbb{V}_{\ell}$ descent to the canonical model $Sh_{n, \mathbb{Q}}$. Similarly, the algebraic vector bundle with connection $(\mathcal{V}_{dR}, \nabla)$ associated with $\mathbb{V}$ descents to $Sh_{n, \mathbb{Q}}$. The period map (\ref{periodQ}) is defined in such a way that the variation of Hodge structure $\mathcal{P}_B$ on $M_{2d,n, \mathbb{C}}$ is just the pullback $\Phi^*\mathbb{V}$ of the variation of Hodge structure $\mathbb{V}$ on $Sh_{n,\mathbb{C}}$, and similarly $\mathcal{P}_{\ell} = \Phi_{\mathbb{Q}}^{*}\mathbb{V}_{\ell}$ for all $\ell$.

\subsection{The GSpin cover}

The orthogonal Shimura variety considered in the previous section has a natural cover by a Shimura variety of Hodge type associated to the group of spinor similitudes of the lattice $\Lambda_{d, \mathbb{Q}}$.
To define this group, denote by $$CL(\Lambda_{d,\mathbb{Q}}):= \left( \bigoplus_{n} \Lambda_{d, \mathbb{Q}}^{\otimes n} \right) / (v \otimes v - \langle v, v\rangle)$$ the Clifford algebra of the lattice $\Lambda_{d, \mathbb{Q}}$. It comes equipped with a $\mathbb{Z}/2\mathbb{Z}$-grading, giving a decomposition $CL(\Lambda_{d, \mathbb{Q}}) = CL(\Lambda_{d, \mathbb{Q}})^+ \oplus CL(\Lambda_{d, \mathbb{Q}})^-$.
\begin{mydef}
We define $\tilde{G}$ to be the group of spinor similitudes of the lattice $\Lambda_{d, \mathbb{Q}}$, i.e. the algebraic group over $\mathbb{Q}$ defined by the functor that sends a $\mathbb{Q}$-algebra $R$ to
$$\tilde{G}(R):= GSpin(\Lambda_{d, \mathbb{Q}})(R) = \{g \in (CL(\Lambda_{d, \mathbb{Q}})^+ \otimes_{\mathbb{Q}} R)^{\times} \, | \, g \Lambda_{d,R} g^{-1} \subset \Lambda_{d, R} \}.$$
\end{mydef}
The action of $\tilde{G}$ on $\Lambda_{d, \mathbb{Q}}$ by conjugation induces a short exact sequence
$$1 \to \mathbb{G}_m \to \tilde{G} \overset{ad}{\to} G \to 1,$$
which identifies $G$ with the adjoint group of $\tilde{G}$.

The datum $(\tilde{G}, \mathcal{D})$ forms a Shimura datum.
For a compact open subgroup $\tilde{K} \subset \tilde{G}(\mathbb{A}_f)$ we denote the associated Shimura variety by
\begin{equation} \label{ShSpin} \widetilde{Sh}_{\tilde{K}, \mathbb{C}} := \tilde{G}(\mathbb{Q}) \backslash \mathcal{D} \times \tilde{G}(\mathbb{A}_f) / \tilde{K}. \end{equation}
Again, this Shimura variety has a canonical model $\widetilde{Sh}_{\tilde{K}, \mathbb{Q}}$ over $\mathbb{Q}$.
For any $\tilde{K}$ such that $\tilde{K}^{ad} \subset K$, the morphism $\tilde{G} \to G$ gives rise to a morphism of Shimura varieties
$$ad: \widetilde{Sh}_{\tilde{K}, \mathbb{Q}} \to Sh_{K, \mathbb{Q}}.$$
We define $$\tilde{\mathcal{G}}(\hat{\mathbb{Z}}) := \tilde{G}(\mathbb{A}_f) \cap CL(\Lambda_{d, \hat{\mathbb{Z}}})^{\times},$$
introduce the congruence subgroups $$\tilde{K}_n:= \{g \in \tilde{\mathcal{G}}(\hat{\mathbb{Z}}) |\, g \equiv 1 \mod n \} $$
and write $\widetilde{Sh}_{n, \mathbb{Q}} := \widetilde{Sh}_{\tilde{K}_n, \mathbb{Q}}$.
One can check (compare \cite{MP2}, Lemma 2.6) that $\tilde{K}_n^{ad} \subset K_n$, and thus we get a morphism of Shimura varieties
$$ad: \widetilde{Sh}_{n, \mathbb{Q}} \to Sh_{n, \mathbb{Q}},$$
which is a finite étale cover (cf. \cite{MP}, 3.4).

It is well-known (compare \cite{MP2}, 3.5) that the Shimura variety $\widetilde{Sh}_{\tilde{K}, \mathbb{Q}}$ is of Hodge type, i.e. there exists a symplectic form on the Clifford algebra $CL(\Lambda_{d,\mathbb{Q}})$ such that the action of $\tilde{G}$ on $CL(\Lambda_{d, \mathbb{Q}})$ factors through the corresponding general symplectic group $GSp$, defining an embedding of Hodge data $(\tilde{G}, \mathcal{D}) \subset (GSp, \mathbb{H})$ into a Siegel Shimura datum.
For any compact open subgroup $\tilde{K} \subset \tilde{G}(\mathbb{A}_f) $, there is a compact open subgroup $\bar{K} \subset GSp(\mathbb{A}_f)$ such that the induced morphism of Shimura varieties
$$ \widetilde{Sh}_{\tilde{K}, \mathbb{Q}} \hookrightarrow Sh(GSp, \mathbb{H})_{\bar{K}, \mathbb{Q}}$$
is a closed immersion of the GSpin Shimura variety into a Siegel modular variety (cf. \cite{DeligneShimura}, Proposition 1.15).

Via its moduli description, the Siegel modular variety $Sh(GSp, \mathbb{H})_{\bar{K}}$ carries a univer\-sal family of abelian varieties.
By restriction, the Shimura variety $\widetilde{Sh}_{\tilde{K}, \mathbb{Q}}$ will also carry a family of abelian varieties $f: \mathfrak{A} \to \widetilde{Sh}_{\tilde{K}, \mathbb{Q}} $ with certain Hodge tensors, as we will now explain.

The representation $\tilde{G} \to GL(CL(\Lambda_{d, \mathbb{Q}}))$ and the lattice $CL(\Lambda_d)$ define a $\mathbb{Z}$-variation of Hodge structure $\tilde{\mathbb{V}}$ over $\widetilde{Sh}_{\tilde{K}, \mathbb{C}}$ which can naturally be identified with the variation of Hodge structure $R^1f^{an}_{\mathbb{C},*}\underline{\mathbb{Z}} $ arising from the family of abelian varieties $f$.
Similarly, we get $\ell$-adic étale local systems $\tilde{\mathbb{V}}_{\ell}$ which descent to $\widetilde{Sh}_{\tilde{K}, \mathbb{Q}}$ and are identified with
$ R^1f_{\acute{e}t, *} \underline{\mathbb{Z}_{\ell}}.$

Hodge type Shimura varieties usually carry additional generic Hodge tensors coming from the tensors fixed by the reductive group defining the Shimura datum.
In our case, one can show that there is a natural idempotent operator on $CL(\Lambda_{d})^{\otimes (1,1)} := CL(\Lambda_d) \otimes CL(\Lambda_d)^{\vee}$ fixed by $\tilde{G}$ (cf. \cite{MP2}, 3.4).
This gives rise to an idempotent operator $\pi$ on $\tilde{\mathbb{V}}^{\otimes (1,1)} = \tilde{\mathbb{V}} \otimes \tilde{\mathbb{V}}^{\vee}$ which is absolute Hodge, and thus to idempotent operators $\pi_{\ell}$ on $\tilde{\mathbb{V}}_{\ell}^{\otimes (1,1)}$.

We denote by $L_B = \mathrm{im \,} \pi \subset \tilde{\mathbb{V}}^{\otimes (1,1)} $ the sub-local system given by the image of this operator, and similarly $L_{\ell} = \mathrm{im \,} \pi_{\ell} \subset \tilde{\mathbb{V}}_{\ell}^{\otimes (1,1)} $.
One can show (\cite{MP}, 3.4) that there is a natural identification
$$L_B = ad_{\mathbb{C}}^{*}\mathbb{V},$$
and similarly, $L_{\ell} = ad^{*}\mathbb{V}_{\ell}$.

We end the section with a brief reminder of the Kuga-Satake construction from a Hodge theoretic point of view.
As suggested by the period map $\Phi$ and the above discussion of Shimura varieties, there is a closed relation between K3 surfaces and abelian varieties.
Namely, for a point $x \in M_{2d,n}(\mathbb{C})$ corresponding to a polarized K3 surface $(X, \xi)$, the image $\Phi(x) \in Sh_{n}(\mathbb{C})$ has a lift $\widetilde{\Phi(x)} \in \widetilde{Sh}_{n}(\mathbb{C})$ to the GSpin Shimura variety. The fiber over $\widetilde{\Phi(x)}$ of the universal abelian variety $\mathfrak{A}$ defines a complex abelian variety $A$ that is called the \emph{Kuga-Satake abelian variety} of $X$. It does not depend on the choice of the level-$n$ structure.

We have the following description of the morphism of real algebraic groups $\rho_A: \mathbb{S} \to GL(H_B^1(A, \mathbb{R}))$ defining the Hodge structure on $H_B^1(A, \mathbb{Q})$.
Denote by $\rho_X: \mathbb{S} \to SO(P_B^2(X,\mathbb{R}))$ the representation of the Deligne torus corresponding to the Hodge structure $P_B^2(X, \mathbb{Q})(1)$ of weight zero (the Tate twist is essential here).
There exists a lift
\begin{equation}\label{KS} \xymatrix{\mathbb{S} \ar[r]^-{\tilde{\rho}_X} \ar[rd]_-{\rho_X} & GSpin(P_B^2(X, \mathbb{R})) \ar[d]^-{ad} \\
& SO(P_B^2(X,\mathbb{R})),} \end{equation}
which via the natural action of $GSpin(P_B^2(X, \mathbb{R}))$ defines a Hodge structure on the Clifford algebra $H_B^1(A, \mathbb{Q}) = CL(P_B^2(X, \mathbb{Q}))$. The lift in (\ref{KS}) becomes unique after requiring that this Hodge structure is of weight one (cf. \cite{K3global}, Remark 4.2.1).

\section{Rapoport-Zink uniformization of Shimura varieties of abelian type}

We survey the theory \cite{Intmod} of integral canonical models for abelian type Shimura varieties at hyperspecial level, and their relation to local Shimura varieties as defined in \cite{Berkeley} via the Rapoport-Zink uniformization established in \cite{Xu}.

\subsection{Integral canonical models}\label{Subsectionintcanmodels}

Let $(G,\mathcal{D})$ be a Shimura datum of abelian type, cf. (\cite{Milnemotives}, Definition 3.2). For a compact open subgroup $K \subset G(\mathbb{A}_f)$ we denote by
$$Sh(G,\mathcal{D})_K= G(\mathbb{Q}) \backslash \mathcal{D} \times G(\mathbb{A}_f) /K$$ the Shimura variety corresponding to the Shimura datum $(G,\mathcal{D})$ and the level $K$.
It is well-known (see \cite{Delignemodels}, 2.7.21) that $Sh(G,\mathcal{D})_K$ has the structure of a quasi-projective algebraic variety defined over a number field $E$, called the reflex field of the Shimura datum $(G, \mathcal{D})$.
Assume that $K_p$ is a hyperspecial subgroup of $G(\mathbb{Q}_p)$, i.e. that there exists a reductive group $\mathcal{G}$ over $\mathbb{Z}_p$ with generic fiber $G_{\mathbb{Q}_p}$ such that $K_p = \mathcal{G}(\mathbb{Z}_p)$.
We look at the pro-scheme
\begin{equation}\label{system} Sh(G,\mathcal{D})_{K_p}:= \varprojlim_{K^p} Sh(G,\mathcal{D})_{K_pK^p} \end{equation}
where $K^p$ runs through the compact open subgroups of $G(\mathbb{A}_f^p)$.

Let $\mathcal{O}$ denote the ring of integers of $E$ and $\mathcal{O}_{(v)}$ the localization at a prime $v$ dividing $p$.
Langlands suggested that the system in (\ref{system}) extends to a pro-system $\mathscr{S}(G,\mathcal{D})_{K_p K^p}$ of smooth schemes over the ring of integers $\mathcal{O}_{(v)}$.

Milne formulated the following extension property in order to uniquely characterize this extension.
\begin{mydef}[\cite{Intmod}, 2.3.7; \cite{Milnepoints}, Definition 2.5]
The pro-scheme $\mathscr{S}(G,\mathcal{D})_{K_p}$ is said to satisfy the \emph{extension property} if the following criterion holds:
For every regular and formally smooth $\mathcal{O}_{(v)}$-scheme $\mathcal{R}$ with generic fiber $R$, any morphism $R \to Sh(G,\mathcal{D})_{K_p}$ over $E$ extends to $\mathcal{R} \to \mathscr{S}(G,\mathcal{D})_{K_p}$.
An \emph{integral canonical model} for $Sh(G,\mathcal{D})_{K_p}$ is a smooth model of $Sh(G,\mathcal{D})_{K_p}$ over $\mathcal{O}_{(v)}$ which satisfies the extension property.
\end{mydef}

Kisin proved the existence of integral canonical models in great generality.

\begin{thm}[\cite{Intmod}]\label{integralmodel}
Suppose that $p>2$. Let $(G,\mathcal{D})$ be a Shimura datum of abelian type and $K_p$ a hyperspecial subgroup of $G(\mathbb{Q}_p)$.
Let $\mathcal{O}$ denote the ring of integers of the reflex field $E$ and $\mathcal{O}_{(v)}$ the localization at a prime $v \mid p$.
Then there exists a unique pro-system $\mathscr{S}(G,\mathcal{D})_{K_pK^p}$ over $\mathcal{O}_{(v)}$ such that the pro-scheme $$\mathscr{S}(G,\mathcal{D})_{K_p}:= \varprojlim_{K^p} \mathscr{S}(G,\mathcal{D})_{K_pK^p}$$ is an integral canonical model of $Sh(G,\mathcal{D})_{K_p}$.
\end{thm}

We turn back to the Shimura variety considered in Section \ref{Shsection}, where $G$ is the special orthogonal group $SO(\Lambda_{d, \mathbb{Q}})$.
For $p \nmid 2d$, the group $K_p = \mathcal{G}(\mathbb{Z}_p)$ as defined in (\ref{integralgroup}) in Section \ref{sectioncoho} is a hyperspecial subgroup of $G(\mathbb{Q}_p)$.
Assuming in addition that $p \nmid n$, Theorem \ref{integralmodel} provides us with an integral canonical model $\mathscr{S}_{n}$ over $\mathbb{Z}_{(p)}$ of the Shimura variety $Sh_{n, \mathbb{Q}}$ defined in (\ref{ShSO}).
Similarly, for $p \nmid 2d$ and $p \nmid n$, the GSpin Shimura variety $\widetilde{Sh}_{n, \mathbb{Q}}$ considered in (\ref{ShSpin}) admits an integral canonical model $\widetilde{\mathscr{S}}_{n}$ over $\mathbb{Z}_{(p)}$.

By construction, the universal abelian scheme $\mathfrak{A}$ over $\widetilde{Sh}_{n}$ extends to a universal abelian scheme $f: \mathcal{A} \to
\widetilde{\mathscr{S}}_{n}$.
Similarly, the $\ell$-adic étale local systems $\tilde{\mathbb{V}}_{\ell}$ for $\ell \not= p$ extend to $\widetilde{\mathscr{S}}_{n}$.
In characteristic $p$, there is a new feature: the $F$-crystal
$$ \tilde{\mathbb{V}}_{cris} := R^1 \overline{f}_{*,cris} \mathcal{O}_{\overline{\mathcal{A}}/W}$$
over $\widetilde{\mathscr{S}}_{n, \bar{\mathbb{F}}_p}$.
Here $\overline{f}: \overline{\mathcal{A}} \to \widetilde{\mathscr{S}}_{n, \bar{\mathbb{F}}_p}$ denotes the base change of $f$ to $\bar{\mathbb{F}}_p$.

Since $\widetilde{\mathscr{S}}_{n}$ is normal, the idempotent operator $\pi_{\ell}$ on $\tilde{\mathbb{V}}_{\ell}^{\otimes (1,1)}$ over $\widetilde{Sh}_{n}$ extends to an operator over $\widetilde{\mathscr{S}}_{n}$.
In addition, there is a natural $F$-invariant idempotent operator $\pi_{cris}$ on the crystal $\tilde{\mathbb{V}}_{cris}^{\otimes (1,1)}$ (cf. \cite{MP2}, 4.14).

As a consequence, if $F$ is an algebraically closed field of characteristic $p$ and $x \in \widetilde{\mathscr{S}}_{n}(F)$, the operators $\pi_{\ell}$ and $\pi_{cris}$ define distinguished subspaces $$L_{\ell,x}:=\mathrm{im}(\pi_{\ell,x}) \subset \mathrm{End}(H_{\acute{e}t}^1(\mathcal{A}_x, \mathbb{Z}_{\ell}))$$ for $\ell \not=p$, and
$$L_{cris,x}:=\mathrm{im}(\pi_{cris,x}) \subset \mathrm{End}(H_{cris}^1(\mathcal{A}_x/W(F))).$$

Following Remark \ref{remarkorientation}, assume now that $n = q^{r}$ is sufficiently large and the power of a prime number $q \not=p$. The integral canonical model can be used to extend the period map for the moduli space of polarized K3 surfaces over $\mathbb{Z}_{(p)}$.

\begin{prop}[\cite{RizovCrelle}; \cite{MP}, Corollary 4.15]\label{integralext}
For $p \nmid 2d$, the period map $\Phi_{\mathbb{Q}}: M_{2d,n, \mathbb{Q}} \hookrightarrow Sh_{n, \mathbb{Q}}$ extends to an open immersion
$$ \varphi: M_{2d,n, \mathbb{Z}_{(p)}} \hookrightarrow \mathscr{S}_{n}$$ over $\mathbb{Z}_{(p)}$.
\end{prop}
This follows from the extension property of $\mathscr{S}_{n}$ together with the fact that $M_{2d,n, \mathbb{Z}_{(p)}}$ is smooth if $p \nmid 2d$.

Madapusi-Pera shows that one can even define a period map over $\mathbb{Z}[\frac{1}{2}]$ (cf. \cite{MP}, Proposition 4.7).

\subsection{Local Shimura varieties}

We recall the notion of a local Shimura variety as a local analog of a Shimura variety, following the definition in \cite{Berkeley}.

Let $\mathrm{Perf}_{\bar{\mathbb{F}}_p}$ denote the category of perfectoid spaces over $\bar{\mathbb{F}}_p$, cf. (\cite{Berkeley}, Definition 7.1.2).
For a perfectoid space $T$ in $\mathrm{Perf}_{\bar{\mathbb{F}}_p}$ we denote by $X_{FF,T}$ the associated relative Fargues-Fontaine curve over $T$, see (\cite{Berkeley}, Definition 15.2.6).

Let $G$ be a reductive group over $\mathbb{Q}_p$.

\begin{prop}[\cite{Berkeley}, Theorem 22.5.2] \label{torsors}
There is an equivalence of categories between the category of $G$-torsors on the relative Fargues-Fontaine curve $X_{FF,T}$ which are trivial at each geometric point of $T$ and the category of pro-étale $\underline{G(\mathbb{Q}_p)}$-torsors on $T$.
\end{prop}

Here $\underline{G(\mathbb{Q}_p)}$ denotes the pro-étale sheaf on $\mathrm{Perf}_{\bar{\mathbb{F}}_p}$ attached to the locally profinite group $G(\mathbb{Q}_p)$.

\begin{mydef}[\cite{Berkeley}, Definition 24.1.1]\textnormal{ }
\begin{enumerate}[(i)]
\item A \emph{local Shimura datum} is a triple $(G, b, \mu)$ consisting of a reductive group $G$ over $\mathbb{Q}_p$, the $G(\bar{\mathbb{Q}}_p)$-conjugacy class of a minuscule cocharacter $\mu: \mathbb{G}_{m, \bar{\mathbb{Q}}_p} \to G_{\bar{\mathbb{Q}}_p}$ and $b \in B(G, \mu)$.
\item A \emph{morphism of local Shimura data} from $(G,b,\mu)$ to $(G', b', \mu')$ is a morphism of reductive groups $G \to G'$ over $\mathbb{Q}_p$ that maps the conjugacy class of $\mu$ to the conjugacy class of $\mu'$ and $b$ to $b'$.
\end{enumerate}
\end{mydef}
The reflex field of the local Shimura datum is the field of definition $E$ of the conjugacy class of $\mu$.

There is a natural way to associate a $G$-torsor $\mathcal{E}_b$ over $X_{FF,T}$ to an element $b \in B(G)$, see (\cite{Curve}, Theorem 4.1).

\begin{mydef}[\cite{Berkeley}, Proposition 23.3.1]
Let $(G, b, \mu)$ be a local Shimura datum and $K \subset G(\mathbb{Q}_p)$ a compact open subgroup.
Define the functor $Sht(G, b, \mu)_K$ on $Perf_{\bar{\mathbb{F}}_p}$ by sending $T \in Perf_{\bar{\mathbb{F}}_p}$ to the set of quadruples $(T^{\sharp},\mathcal{E}_1, \alpha, \mathbb{P})$ where
\begin{itemize}
\item $T^{\sharp}$ is an untilt of $T$;
\item
$\mathcal{E}_1$ is a $G$-torsor on $X_{FF, T}$ which is trivial at every geometric point of $T$;
\item
$\alpha$ is an isomorphism of $G$-torsors $$\alpha: \left. \mathcal{E}_1\right|_{X_{FF,T} \setminus T^{\sharp}} \cong \left. \mathcal{E}_b\right|_{X_{FF,T} \setminus T^{\sharp}} $$ which is meromorphic along $T^{\sharp}$ and bounded by $\mu$;
\item
$\mathbb{P}$ is a pro-étale $\underline{K}$-torsor on $T$ such that $\mathbb{P} \times^{\underline{K}} \underline{G(\mathbb{Q}_p)}$ is the pro-étale $\underline{G(\mathbb{Q}_p)}$-torsor corresponding to $\mathcal{E}_1$ under Proposition \ref{torsors}.
\end{itemize}
\end{mydef}
Scholze shows in (\cite{Berkeley}, Remark 23.3.4) that $Sht(G, b, \mu)_K$ is a diamond over $(\mathrm{Spa \,} \breve{E})^{\diamond}$,
where $\breve{E}$ denotes the maximal unramified extension of $E$.

Note that $Sht(G,b,\mu)_K$ has a natural action of the group $$J_b(\mathbb{Q}_p) =\{g \in G(\breve{\mathbb{Q}}_p) \,|\, \sigma(g)^{-1} b g = b \}.$$
Let $\mathcal{F}(G,\mu) = G_{\breve{E}} / P_{\mu}$ be the flag variety associated with the pair $(G, \mu)$, i.e. $$P_{\mu}= \{g \in G_{\breve{E}} \,|\, \lim_{t \to \infty} \mu(t) g \mu(t)^{-1} \,\,\textnormal{exists}\}. $$
In (\cite{Berkeley}, 24.1) one can find the construction of an étale period map $$\pi_{GM}: Sht(G, b, \mu)_K \to \mathcal{F}(G, \mu)^{\diamond},$$ where the target is the diamond corresponding to the smooth rigid space $\mathcal{F}(G, \mu)$ over $\breve{E}$. In view of (\cite{Berkeley}, Theorem 10.4.2), it follows that $Sht(G,b,\mu)_K = \mathcal{M}(G,b,\mu)_K^{\diamond}$ is the diamond of a unique smooth rigid space $\mathcal{M}(G,b,\mu)_K$.
Here we view rigid spaces over $\breve{E}$ as adic spaces which are locally of finite type over $\breve{E}$.

A morphism of local Shimura data from $(G,b, \mu)$ to $(G', b', \mu')$ that maps $K$ to $K'$ induces a morphism of rigid spaces
$ \mathcal{M}(G,b,\mu)_K \to \mathcal{M}(G',b',\mu')_{K'}$.

Now we suppose that the group $K$ is hyperspecial, i.e. that there exists a reductive model $\mathcal{G}$ of $G$ over $\mathbb{Z}_p$ such $K = \mathcal{G}(\mathbb{Z}_p)$. In this case one can show that necessarily $\breve{E} = \breve{\mathbb{Q}}_p$.
The affine Deligne-Lusztig set associated with the local Shimura datum $(G, b, \mu)$ is the set
$$ X_{\mu}^G(b) = \{ g \in G(\breve{\mathbb{Q}}_p) / \mathcal{G}(W) \, |\, g^{-1}b \sigma(g) \in \mathcal{G}(W) \mu(p)\mathcal{G}(W)\}.$$
The work of Zhu \cite{Zhu} and Bhatt-Scholze \cite{BS} shows that $X_{\mu}^G(b)$ can be interpreted as the set of $\bar{\mathbb{F}}_p$-points of a perfect scheme $X_{\mu}^G(b)$ over $\bar{\mathbb{F}}_p$, called an affine Deligne-Lusztig variety, equipped with an action of $J_b(\mathbb{Q}_p)$.
Similarly as above, a morphism of local Shimura data from $(G,b, \mu)$ to $(G', b', \mu')$ induces a morphism $X_{\mu}^G(b) \to X_{\mu'}^{G'}(b')$.
There is a $J_b(\mathbb{Q}_p)$-equivariant continuous specialization map of topological spaces $$sp: \left| \mathcal{M}(G,b,\mu)_K\right| \to \left| X_{\mu}^G(b)\right|.$$

In the case of a local Shimura datum $(G,b,\mu)$ of abelian type and a hyperspecial subgroup $K \subset G(\mathbb{Q}_p)$, Shen shows that there exists a (generalized) Rapoport-Zink space $\breve{\mathcal{M}}$ over $W$ that such that the rigid generic fiber of $\breve{\mathcal{M}}$ is $\mathcal{M}(G, b, \mu)_K$ and the perfection of the special fiber is $X_{\mu}^G(b)$.
\begin{thm}[\cite{Xu}, Theorem 1.1]
Suppose that the local Shimura datum $(G,b,\mu)$ is of abelian type and the subgroup $K \subset G(\mathbb{Q}_p)$ is hyperspecial.
There exists a unique formal scheme $\breve{\mathcal{M}}= \breve{\mathcal{M}}(G,b,\mu)$ which is formally smooth and locally formally of finite type over $\mathrm{Spf \,} W$ such that
\begin{enumerate}[(i)]
\item
$\breve{\mathcal{M}}_{\eta}^{rig} = \mathcal{M}(G,b,\mu)_K$;
\item
$\breve{\mathcal{M}}_{\bar{\mathbb{F}}_p}^{perf} = X^G_{\mu}(b)$;
\end{enumerate}
and the map $sp: \left| \mathcal{M}(G,b,\mu)_K\right| \to \left| X_{\mu}^G(b)\right| $ can be identified with the specialization map for the formal scheme $\breve{\mathcal{M}}$. \end{thm}

Here the perfection of a scheme $Y$ over $\mathbb{F}_p$ is defined as the limit $Y^{perf}:= \varprojlim_{\mathrm{Frob}} Y$ over the Frobenius map.

We now turn to the situation of the orthogonal type Shimura variety considered in Section \ref{Shsection}.
Let $G = SO(\Lambda_{d, \mathbb{Q}_p})$ and
\begin{eqnarray*}\mu: \mathbb{G}_{m, \mathbb{Q}_p} & \to & G, \\
t & \mapsto & \left(\begin{array}{rrrrr}t & & & & \\ & 1 & & & \\ & & \ddots & & \\ & & & 1 & \\ & & & & t^{-1} \end{array} \right) . \end{eqnarray*}
Then the flag variety attached to $(G, \mu)$ is $$\mathcal{F} = \mathcal{F}(G,\mu) = \{x \in \mathbb{P}(\Lambda_d \otimes \mathbb{Q}_p) | \langle x,x \rangle=0 \} \subset \mathbb{P}(\Lambda_d \otimes \mathbb{Q}_p).$$
Note that $\mathcal{F} \subset \mathbb{P}(\Lambda_{d,\mathbb{Q}})$ can in fact be defined over $\mathbb{Q}$.

If $b \in B(G,\mu)$ is the unique basic element, then a result of Chen-Fargues-Shen \cite{CFS} asserts that the image of $\mathcal{M}(G,b,\mu)_K$ in $\mathcal{F}(G,\mu)$ is given by the rigid open subspace $\mathcal{F}(G,b,\mu)^{wa}$ with $\mathbb{C}_p$-points
$$ \mathcal{F}(G,b,\mu)^{wa}(\mathbb{C}_p) = \mathcal{F}(\mathbb{C}_p) \setminus \mathcal{F}(\mathbb{Q}_p),$$
the \emph{weakly admissible locus}.
Note the similarity to the complex period domain $\mathcal{D} \subset \mathcal{F}(\mathbb{C})$, which can be identified with a union of connected components of the analytic open subspace $\mathcal{F}(\mathbb{C}) \setminus \mathcal{F}(\mathbb{R})$.

\subsection{Rapoport-Zink uniformization}

We discuss the link between the local Shimura varieties in the previous section and global Shimura varieties.
The guiding principle is that the completion of the integral canonical model of the Shimura variety along the basic locus will be uniformized by a local Shimura variety. This can be viewed as a $p$-adic analog of complex uniformization.

Let $Sh_K:= Sh(G, \mathcal{D})_K$ be a Shimura variety of abelian type with $K_p$ hyperspecial, and let $\mathscr{S}_K$ be the integral canonical model of $Sh_K$ given by Theorem \ref{integralmodel}. Denote by $\overline{\mathscr{S}}_K= \mathscr{S}_K \otimes \bar{\mathbb{F}}_p$ the special fiber of $\mathscr{S}_K$. Recall (\cite{Xu}, Theorem 6.2) that there exists a Newton stratification
$$ \overline{\mathscr{S}}_K = \coprod_{b \in B(G,\mu)} \overline{\mathscr{S}}_K^b $$
such that the Zariski closure of $\overline{\mathscr{S}}_K^b $ in $\overline{\mathscr{S}}_K$ is $\coprod_{b' \le b} \overline{\mathscr{S}}_K^{b'}$.
In particular, if $b \in B(G, \mu)$ is the basic element, then $\overline{\mathscr{S}}_K^b$ is closed.
We remind the reader that to an $\bar{\mathbb{F}}_p$-point $x \in \overline{\mathscr{S}}_K(\bar{\mathbb{F}}_p)$ one can associate a $G$-isocrystal $\mathcal{I}_x$, and then the point $x$ belongs to $\overline{\mathscr{S}}_K^b$ if and only if the isomorphism class of $\mathcal{I}_x$ corresponds to $b \in B(G,\mu)$.
The Langlands-Rapoport conjecture, for Shimura varieties of abelian type now a theorem of Kisin (\cite{Kisin}, Theorem 4.6.7), asserts that we can decompose the set of $\bar{\mathbb{F}}_p$-points of $\overline{\mathscr{S}}_K^b$ further into a disjoint union $$ \overline{\mathscr{S}}_K^b(\bar{\mathbb{F}}_p) = \coprod_{[\phi], b(\phi)=b} \overline{\mathscr{S}}(\phi).$$
Here $\phi: \mathfrak{Q} \to G$ runs through the set of admissible representations of the Langlands-Rapoport gerbe (cf. \cite{Kisin}, 3.3), $[\phi]$ is the associated equivalence class, and $b(\phi)$ the element of $B(G,\mu)$ attached to $\phi$.
If $\mathscr{S}_K$ is of Hodge type, then the set $\overline{\mathscr{S}}_K(\phi)$ corresponds to the isogeny class of an abelian variety over $\bar{\mathbb{F}}_p$ with certain $\ell$-adic and crystalline tensors defined by the group $G$.

In general, for the basic element $b \in B(G,\mu)$ there is a unique representation $\phi$ (up to equivalence) such that $b(\phi)=b$, and thus $\overline{\mathscr{S}}_K^b = \overline{\mathscr{S}}(\phi)$.
We can now begin to formulate the Rapoport-Zink uniformization for the basic locus. Let $\widehat{ {\mathscr{S}_{K,W}}_{/\overline{\mathscr{S}}_K^b}}$ denote the formal scheme over $\mathrm{Spf \,} W $ that is the completion of $\mathscr{S}_{K}$ over $W$ along the closed subscheme $\overline{\mathscr{S}}_K^b$. Let $Sh_K^b\subset Sh_{\breve{\mathbb{Q}}_p}^{rig}$ be the associated rigid generic fiber, which is the tube over the closed subscheme $\overline{\mathscr{S}}_K^b$.
The global Shimura datum $(G, \mathcal{D})$ and the element $b \in B(G,\mu)$ give a local Shimura datum $(G_{\mathbb{Q}_p}, b, \mu)$, and we denote by $\mathcal{M} := \mathcal{M}(G,b,\mu)_{K_p}$ the associated local Shimura variety. Similarly, we denote by $X:= X_{\mu}^G(b)$ the corresponding affine Deligne-Lusztig variety.

\begin{thm}[\cite{Xu}, Corollary 6.14 ] \label{RZuniformization}
Assume $K_p$ hyperspecial and $b \in B(G, \mu)$ basic.
We have isomorphisms
$$ Sh_K^b \overset{\sim}{\rightarrow} I_{\phi}(\mathbb{Q}) \backslash \mathcal{M} \times G(\mathbb{A}_f^p) /K^p,$$
and $$ \overline{\mathscr{S}}_K^{b,perf} \overset{\sim}{\rightarrow} I_{\phi}(\mathbb{Q})\backslash X \times G(\mathbb{A}_f^p)/K^p.$$
\end{thm}
Here $I_{\phi}$ denotes the group of automorphisms of the representation $\phi: \mathfrak{Q} \to G$, i.e. the algebraic group over $\mathbb{Q}$ with points in a $\mathbb{Q}$-algebra $R$ given by
$$I_{\phi}(R) = \{g \in G(R \otimes_{\mathbb{Q}} \bar{\mathbb{Q}}) \, | \, \textnormal{conj}_g \circ \phi = \phi \}.$$

The right hand side of the first isomorphism in Theorem \ref{RZuniformization} can be written as a finite disjoint union $$ I_{\phi}(\mathbb{Q}) \backslash \mathcal{M} \times G(\mathbb{A}_f^p) /K^p = \coprod_g \Gamma_g \backslash \mathcal{M},$$
where $\Gamma_g = I_{\phi}(\mathbb{Q}) \cap gK^p g^{-1}$ and $g$ runs through a set of representatives of the double coset
$ I_{\phi}(\mathbb{Q}) \backslash G(\mathbb{A}_f^p) /K^p $.
\begin{rem}
For Shimura varieties of PEL type this uniformization was established in the book of Rapoport and Zink (cf. \cite{RZ}, Theorem 6.36), and then generalized to Shimura varieties of Hodge type by Kim \cite{Wansu} and Shen \cite{Xu}.
One can formulate similar uniformization results for suitably defined completions along the other strata $\overline{\mathscr{S}}(\phi)$, and even on the level of formal schemes, see (\cite{Xu}, Theorem 6.7).
\end{rem}

We will be interested in the case where $(G, \mathcal{D})$ is the orthogonal type Shimura datum, and $(\tilde{G}, \mathcal{D})$ is the $GSpin$ cover.
In this case, for $p \nmid 2d$ and $p \nmid n$, by construction of the uniformization in \cite{Xu} we have a commuting diagram

$$\xymatrix{ \widetilde{Sh}_{n}^{\tilde{b}} \ar[r]^-{\cong} \ar[d]& I_{\tilde{\phi}}(\mathbb{Q}) \backslash \tilde{\mathcal{M}} \times \tilde{G}(\mathbb{A}_f^p)/\tilde{K}_n^p \ar[d] \\
Sh_n^{b} \ar[r]^-{\cong} & I_{\phi}(\mathbb{Q}) \backslash \mathcal{M} \times G(\mathbb{A}_f^p)/K_n^p,
}$$
and similarly for the perfection of $ \overline{\mathscr{S}}_n^{b}$.

\section{The $p$-adic period map}

We use Rapoport-Zink uniformization to construct a $p$-adic period map for the moduli space of polarized K3 surfaces with supersingular reduction. For the rest of this paper, we assume $p \nmid 2d$, that $n = q^{r}$ is the power of a prime number $q \not=p$ and $n$ is sufficiently large, see Remark \ref{remarkorientation}.

\subsection{The Kuga-Satake construction in positive characteristic}

As observed by Madapusi-Pera \cite{MP} (see also \cite{RizovCrelle}), the period map $$\varphi: M_{2d,n, \mathbb{Z}_{(p)}} \to \mathscr{S}_{n}$$ over $\mathbb{Z}_{(p)}$
provides a possibility to generalize the classical Kuga-Satake construction to positive characteristic.
Namely, given an algebraically closed field $F$ of characteristic $p \ge 3$ and a point $x \in M_{2d,n, \mathbb{Z}_{(p)}}(F)$ corresponding to a polarized K3 surface $(X, \xi)$ of degree $2d$ over $F$ with a level structure $[\eta]$, we can choose a lift $\widetilde{\varphi(x)} \in \widetilde{\mathscr{S}}_{n}(F)$ of the image $\varphi(x)$.
The moduli interpretation of $\widetilde{\mathscr{S}}_{n}$ provides us with an abelian variety $A$ over $F$ (with an action of the Clifford algebra $CL(\Lambda_d)$) that we will call a Kuga-Satake abelian variety attached to $(X, \xi)$. As described in Section \ref{Subsectionintcanmodels}, it comes with distinguished subspaces $L_{\ell, A} \subset \mathrm{End}(H_{\acute{e}t}^1(A, \mathbb{Z}_{\ell}))$ and $L_{cris, A} \subset \mathrm{End}(H_{cris}^1(A/W(F)))$.
The following properties of the Kuga-Satake abelian variety were established and cru\-cially used in Madapusi-Pera's work on the Tate conjecture for K3 surfaces over finitely generated fields of characteristic $p \ge 3$. We recall them here for the reader's convenience.

\begin{thm}[\cite{MP}, Theorem 4.17]\label{Madapusi-Pera}
Let $(X,\xi)$ be a polarized K3 surface over $F$ of degree $2d$ and $A$ a Kuga-Satake abelian variety attached to $(X,\xi)$.
\begin{enumerate}[(i)]
\item \label{one} There exists an isomorphism of $\mathbb{Z}_{\ell}$-modules $$H_{\acute{e}t}^1(A, \mathbb{Z}_{\ell}) \cong CL(P^2_{\acute{e}t}(X, \mathbb{Z}_{\ell})),$$ as well as an isomorphism of $W(F)$-modules
$$ H_{cris}^1(A/W(F)) \cong CL(P^2_{cris}(X/W(F))).$$
\item \label{naturalsubspaces}
The natural action by left translation on the Clifford algebra in (\ref{one}) induces an embedding $$P^2_{\acute{e}t}(X, \mathbb{Z}_{\ell}) \subset \mathrm{End}(H_{\acute{e}t}^1(A, \mathbb{Z}_{\ell}))$$ which identifies $P^2_{\acute{e}t}(X, \mathbb{Z}_{\ell})$ with the distinguished subspace $$L_{\ell, A} \subset \mathrm{End}(H_{\acute{e}t}^1(A, \mathbb{Z}_{\ell})).$$
Similarly, the $F$-equivariant embedding $$P^2_{cris}(X/W(F))(1) \subset \mathrm{End}(H_{cris}^1(A/W(F)))$$ identifies
$ P^2_{cris}(X/W(F))(1) $ with the distinguished subspace $$L_{cris,A} \subset \mathrm{End}(H_{cris}^1(A/W(F))).$$
\item \label{NSKS}
Let $L(A) \subset \mathrm{End}(A)$ be the subspace of \emph{special endomorphisms}, i.e. endomorphisms whose cohomological realization lies in $P^2_{\acute{e}t}(X, \mathbb{Z}_{\ell})$ under the embedding in (\ref{naturalsubspaces}) for all $\ell \not=p$, as well as in $P^2_{cris}(X/W(F))(1)$.
Then we have a natural identification
$$ NS(X) \supset \langle\xi\rangle^{\perp} \cong L(A).$$
\end{enumerate}
\end{thm}

One can extend the Kuga-Satake construction to quasi-polarized K3 surfaces, and the analog of Theorem \ref{Madapusi-Pera} holds, see (\cite{Isogenies}, Proposition 3.11).

\begin{rem}
Note that for $f \in L(A)$ the composition $f \circ f \in \mathrm{End}(A)$ is in fact a scalar, equipping $L(A)$ with a quadratic form. The identification in (\ref{NSKS}) is even an isometry.
\end{rem}

We provide an alternative viewpoint on the Kuga-Satake construction in the case $F = \bar{\mathbb{F}}_p$ . As discussed above, using the Langlands-Rapoport conjecture for Shimura varieties of abelian type (\cite{Kisin}, Theorem 4.6.7) we can attach to the point $\varphi(x) \in \mathscr{S}_{n}(\bar{\mathbb{F}}_p)$ a representation $\phi: \mathfrak{Q} \to G$ of the Langlands-Rapoport gerbe $\mathfrak{Q}$.
Similarly, the lift $\widetilde{\varphi(x)} \in \widetilde{\mathscr{S}}_{n}(\bar{\mathbb{F}}_p) $ gives rise to a representation $\tilde{\phi}: \mathfrak{Q} \to \tilde{G}$, making the diagram
$$\xymatrix{\mathfrak{Q} \ar[r]^{\tilde{\phi}} \ar[rd]_{\phi} & \tilde{G} \ar[d] \\
& G} $$
commute.
The analogy with the diagram (\ref{KS}) justifies calling this lift a Kuga-Satake abelian variety.
Via the conjectural interpretation of the gerbe $\mathfrak{Q}$ as the fundamental gerbe of the category of motives over $\bar{\mathbb{F}}_p$ (cf. \cite{Milnepoints}, Theorem 3.25), the Kuga-Satake construction should correspond to lifting the $G$-motive corresponding to the K3 surface $(X, \xi)$ to a $\tilde{G}$-motive over $\bar{\mathbb{F}}_p$, which via the natural action of $\tilde{G}$ on $CL(\Lambda_{d, \mathbb{Q}})$ gives the motive of the Kuga-Satake abelian variety. The process of attaching a Kuga-Satake abelian variety to a K3 surface should thus be a motivic construction, which justifies the statement on algebraic cycles in Theorem \ref{Madapusi-Pera}(\ref{NSKS}).

\subsection{Isogenies of K3 surfaces}

Let $x \in \overline{M}_{2d,ss,n}(\bar{\mathbb{F}}_p)$ be a point and $(X, \xi)$ the associated supersingular K3 surface with level structure.
Denote by $\phi: \mathfrak{Q} \to G$ the representation of the Langlands-Rapoport gerbe corresponding to the point $\varphi(x) \in \mathscr{S}_{n}(\bar{\mathbb{F}}_p)$.
In this section we compute the automorphism group $I_{\phi}$ of $\phi$ appearing in the Rapoport-Zink uniformization in Theorem \ref{RZuniformization} concretely in terms of the polarized K3 surface $(X, \xi)$.

Recall that for a representation $\phi: \mathfrak{Q} \to G$, we denote by $I_{\phi}$ the algebraic group over $\mathbb{Q}$ with points in a $\mathbb{Q}$-algebra $R$ given by
$$I_{\phi}(R) = \{g \in G(R \otimes_{\mathbb{Q}} \bar{\mathbb{Q}}) \, | \, \textnormal{conj}_g \circ \phi = \phi \}.$$

Now let $\tilde{\phi}: \mathfrak{Q} \to \tilde{G} $ be the representation of the Langlands-Rapoport gerbe attached to a lift $\widetilde{\varphi(x)} \in \widetilde{\mathscr{S}}_{n}(\bar{\mathbb{F}}_p)$, and $A$ the associated Kuga-Satake abelian variety over $\bar{\mathbb{F}}_p$ with distinguished subspaces $L_{cris,A} \subset \mathrm{End}(H^1_{cris}(A/W))$ and $L_{\ell,A} \subset \mathrm{End}(H^1_{\acute{e}t}(A, \mathbb{Z}_{\ell}))$ for $\ell \not=p$.

Fix isometries $\Lambda_{d, \hat{\mathbb{Z}}^p} \cong P^2_{\acute{e}t}(X, \hat{\mathbb{Z}}^p)$ and $\Lambda_{d, \mathbb{Q}_p} \cong (P^2_{cris}(X/W)[\frac{1}{p}])^{\varphi=p}$, inducing iso\-morphisms $H^1_{\acute{e}t}(A, \hat{\mathbb{Z}}^p) \cong CL(\Lambda_{d, \hat{\mathbb{Z}}^p})$ and $H^1_{cris}(A/W)[\frac{1}{p}] \cong CL(\Lambda_{d,\breve{\mathbb{Q}}_p})$ by Theorem \ref{Madapusi-Pera}.
If we denote by $N_{d} = \langle \xi \rangle^{\perp} \subset NS(X)$ the primitive part of the Neron-Severi lattice of $X$, the Chern class induces
isometries \begin{equation}\label{isometries} N_{d, \hat{\mathbb{Z}}^p} \cong \Lambda_{d,\hat{\mathbb{Z}}^p}, \, \, \, \, N_{d, \mathbb{Q}_p} \cong \Lambda_{d,\mathbb{Q}_p}. \end{equation}
As the element $b(\phi) \in B(G, \mu)$ corresponding to $\phi$ is basic, we have $I_{\phi, \mathbb{A}_f^p} = G_{\mathbb{A}_f^p}$ and $I_{\phi,\mathbb{Q}_p} = J_b $.

\begin{prop}\label{isogenygroup}
There is an isomorphism of groups $$I_{\phi} \cong SO(N_{d,\mathbb{Q}})$$
compatible with the isomorphisms $G_{\mathbb{A}_f^p} \cong SO(N_{d, \mathbb{A}_f^p}) $ and $ J_b \cong SO(N_{d, \mathbb{Q}_p})$ provided by the isometries in (\ref{isometries}). \end{prop}

\begin{proof}
It follows from the definition of the automorphism group that the short exact sequence
$$ 1 \to \mathbb{G}_{m,\mathbb{Q}} \to \tilde{G} \overset{ad}{\to} G \to 1 $$
induces a short exact sequence
$$ 1 \to \mathbb{G}_{m, \mathbb{Q}} \to I_{\tilde{\phi}} \to I_{\phi} \to 1.$$
We first construct a morphism $I_{\tilde{\phi}} \to SO(N_{d,\mathbb{Q}})$.

Denoting by $\mathrm{Aut}^0(A)$ the $\mathbb{Q}$-group scheme of units in $\mathrm{End}^0(A):=\mathrm{End}(A) \otimes \mathbb{Q}$, it is shown in \cite{Kisin} that $I_{\tilde{\phi}}$ is the largest closed $\mathbb{Q}$-subgroup of $\mathrm{Aut}^0(A)$ which maps to $\tilde{G}_{\mathbb{Q}_{\ell}}= GSpin(P_{\acute{e}t}^2(X, \mathbb{Q}_{\ell}))$ for all $\ell \not=p$, and to $\tilde{G}_{\breve{\mathbb{Q}}_p}=GSpin(P^2_{cris}(X/W)[\frac{1}{p}])$.
In particular, the action of $I_{\tilde{\phi}}$ by conjugation preserves the subspaces $L_{\ell,A}\otimes_{\mathbb{Z}_{\ell}} \mathbb{Q}_{\ell} \subset \mathrm{End}(H_{\acute{e}t}^1(A, \mathbb{Q}_{\ell}))$ for all $\ell \not= p$ and the subspace $L_{cris,A}[\frac{1}{p}] \subset \mathrm{End}(H^1_{cris}(A/W)[\frac{1}{p}])$.
By Theorem \ref{Madapusi-Pera}(\ref{NSKS}) we have an identification $SO(N_{d,\mathbb{Q}}) = SO(L(A)_{\mathbb{Q}})$, where $L(A) \subset \mathrm{End}(A)$ denotes the set of special endomorphisms.
There is a natural map $$\mathrm{Aut}^0(A) \to \mathrm{Aut}(\mathrm{End}^0(A)), g \mapsto \mathrm{conj}_g$$ and we need to show that the restriction to the subgroup $I_{\tilde{\phi}} \subset \mathrm{Aut}^0(A)$ maps to $SO(L(A)_{\mathbb{Q}})$.
For this, recall that $L(A) \subset \mathrm{End}(A)$ is the subset of endomorphisms whose $\ell$-adic realizations lie in $L_{\ell, A}$ for all $\ell \not= p$ and whose crystalline realization lies in $L_{cris, A}$.
Since $I_{\tilde{\phi}}$ preserves the subspaces $L_{\ell,A}\otimes_{\mathbb{Z}_{\ell}} \mathbb{Q}_{\ell}$ and $L_{cris,A}[\frac{1}{p}]$ we get the desired map to $\mathrm{Aut}(L(A)_{\mathbb{Q}})$.
We see that it in fact maps to $SO(L(A)_{\mathbb{Q}})$ since $$\mathrm{conj}_g(f)\circ \mathrm{conj}_g(f) = \mathrm{conj}_g(f\circ f) =f \circ f,$$ where the last equality holds since $ f\circ f$ is a scalar.
In particular we get a sequence
\begin{equation} \label{ses} 1 \to \mathbb{G}_{m, \mathbb{Q}} \to I_{\tilde{\phi}} \to SO(N_{d,\mathbb{Q}}) \to 1, \end{equation}
which after tensoring with $\mathbb{Q}_{\ell}$ for any $\ell \not= p$ becomes the short exact sequence
\begin{equation} \label{compatibility} 1 \to \mathbb{G}_{m, \mathbb{Q}_{\ell}} \to \tilde{G}_{\mathbb{Q}_{\ell}} \to G_{\mathbb{Q}_{\ell}} \to 1. \end{equation}
We conclude that (\ref{ses}) is a short exact sequence, hence we get an isomorphism $I_{\phi} \cong SO(N_{d,\mathbb{Q}})$.
The compatibility assertions follow from the sequence (\ref{compatibility}) for all $\ell \not=p$ and the analogous sequence at $p$.
\end{proof}

Since $N_{d, \mathbb{R}}$ is negative definite, the real group $SO(N_{d, \mathbb{R}})$ is compact. Therefore the image $I_{\phi}(\mathbb{Q}) \subset J_b(\mathbb{Q}_p) \times G(\mathbb{A}_f^p)$ is a discrete subgroup.
As a consequence, for a compact open $K^p \subset G(\mathbb{A}_f^p)$, the intersection $I_{\phi}(\mathbb{Q}) \cap K^p \subset J_b(\mathbb{Q}_p)$ is a discrete subgroup.
We emphasize that by (\cite{K3crystals}, Theorem 7.4.1) the isomorphism class of the lattice $NS(X)_{\mathbb{Q}}$ does not depend on the choice of the supersingular K3 surface $X$.

For a general $x \in M_{2d,n, \mathbb{Z}_{(p)}}(\bar{\mathbb{F}}_p)$ with associated Langlands-Rapoport parameter $\phi$, one can use the results of \cite{Isogenies} to show that under the assumption that $p > 18d+4$, the group $I_{\phi}$ is the group of self-isogenies of the corresponding polarized K3 surface $(X, \xi)$ in the sense of (\cite{Isogenies}, Definition 1.1).

\subsection{The $p$-adic period map}\label{defpadicper}

In this section we define a $p$-adic period map for the completion of $M_{2d,n,W}$ along the supersingular locus. The starting point of the construction is the integral period map
$ \varphi: M_{2d,n, W} \hookrightarrow \mathscr{S}_{n, W}$ defined in Proposition \ref{integralext}.
Under this period map, the basic locus $\overline{\mathscr{S}}_{n}^b \subset \overline{\mathscr{S}}_{n}$ pulls back to the supersingular locus
$$\overline{M}_{2d,ss,n} \subset \overline{M}_{2d,n} := M_{2d,n, \bar{\mathbb{F}}_p}.$$
To see this, note that for $y \in \overline{M}_{2d,n}(\bar{\mathbb{F}}_p)$ corresponding to a polarized K3 surface $(X, \xi)$ over $\bar{\mathbb{F}}_p$, the $G$-isocrystal attached to the point $\varphi(y) \in \overline{\mathscr{S}}_n(\bar{\mathbb{F}}_p)$ is precisely the $G$-isocrystal $P_{cris}^2(X/W)[\frac{1}{p}](1)$ as studied in Section \ref{K3pos}. We refer to (\cite{MP}, Lemma 4.9) for a proof of this fact.

Denote by $\breve{M}_{2d,ss,n,W} := \widehat{{M_{2d,n,W}}}_{/\overline{M}_{2d,ss,n}} $ the formal scheme over $\mathrm{Spf \,} W$ which is the completion of $M_{2d,n,W}$ along the supersingular locus. We denote by $M_{2d,n, \breve{\mathbb{Q}}_{p}}^{ss}$ the rigid analytic generic fiber of this formal scheme.
The completion of $\varphi$ induces a period map
$$ \breve{\varphi}: \breve{M}_{2d,ss,n,W} \hookrightarrow \widehat{ {\mathscr{S}_{n,W}}_{/\overline{\mathscr{S}}_{n}^b}}$$ of formal schemes over $W$, which is still an open immersion.
Together with the Rapoport-Zink uniformization of the Shimura variety $\mathscr{S}_n$ this gives rise to a rigid analytic period map
$$ \varphi^{rig}: (\breve{M}_{2d,ss,n,W} )_{\eta}^{rig} \hookrightarrow Sh_n^b \overset{\sim}{\rightarrow} I_{\phi}(\mathbb{Q}) \backslash \mathcal{M} \times G(\mathbb{A}_f^p) / K_n^p$$
as well as to a period map of perfect schemes
\begin{equation} \label{perconstr} \overline{\varphi}: \overline{M}_{2d,ss,n}^{perf} \hookrightarrow \overline{\mathscr{S}}_n^{b,perf} \overset{\sim}{\rightarrow} I_{\phi}(\mathbb{Q})\backslash X_{\mu}^G(b) \times G(\mathbb{A}_f^p)/K_n^p \end{equation}
over $\bar{\mathbb{F}}_p$.

We now provide a direct description of the period map
\begin{equation}\label{periodmap} \overline{\varphi}: \overline{M}_{2d,ss,n}^{perf} \hookrightarrow I_{\phi}(\mathbb{Q}) \backslash X_{\mu}^G(b) \times G(\mathbb{A}_f^p)/K_n^p \end{equation}
on the level of $\bar{\mathbb{F}}_p$-points.
We first fix an auxiliary supersingular polarized K3 surface $(\mathbb{X}, \mathbb{L})$ and set $N:= NS(\mathbb{X})$. Denoting by $\lambda \in N$ the class of the line bundle $\mathbb{L}$, set $N_d := \langle \lambda \rangle^{\perp} \subset N$. The choice of isometries $\Lambda_{\hat{\mathbb{Z}}^p} \cong H^2_{\acute{e}t}(\mathbb{X}, \hat{\mathbb{Z}}^p)$ and $\Lambda_{\mathbb{Q}_p} \cong (H^2_{cris}(\mathbb{X}/W)[\frac{1}{p}])^{\varphi=p}$ mapping $\lambda$ to $c_{\hat{\mathbb{Z}}^p}(\mathbb{L})$ (respectively $c_{cris}(\mathbb{L})$) allows us to identify \begin{equation}\label{padicidentification} N_{\hat{\mathbb{Z}}^p} \cong \Lambda_{\hat{\mathbb{Z}}^p}, \, \, \, \, N_{\mathbb{Q}_p} \cong \Lambda_{\mathbb{Q}_p} \end{equation}
in a way which preserves $\lambda$.
For $(X, \xi, \eta \mathcal{O}_n(\hat{\mathbb{Z}}^p) )$ a polarized supersingular K3 surface of degree $2d$ over $\bar{\mathbb{F}}_p$ with a level structure $\eta: \Lambda_{\hat{\mathbb{Z}}^p} \overset{\sim}{\to} H_{\acute{e}t}^2(X, \hat{\mathbb{Z}}^p)$ mapping $\lambda$ to $c_{\hat{\mathbb{Z}}^p}(\xi)$, we define its image under the $p$-adic period map as follows.
Choose an isometry $\rho: NS(X)_{\mathbb{Q}} \overset{\sim}{\to} N_{\mathbb{Q}}$ mapping $\xi$ to $\lambda$, as well as an isometry $h: \Lambda_W \overset{\sim}{\to} H^2_{cris}(X/W)$ mapping $\lambda$ to $c_{cris}(\xi)$.
The composition $\rho_{\breve{\mathbb{Q}}_p} \circ h_{\breve{\mathbb{Q}}_p}$ defines an isometry
$$ \Lambda_{\breve{\mathbb{Q}}_p} \overset{h_{\breve{\mathbb{Q}}_p}}{\longrightarrow} H^2_{cris}(X/W)[\frac{1}{p}]\cong NS(X)_{\breve{\mathbb{Q}}_p} \overset{\rho_{\breve{\mathbb{Q}}_p}}{\longrightarrow} N_{\breve{\mathbb{Q}}_p} \cong \Lambda_{\breve{\mathbb{Q}}_p}$$
that preserves $\lambda$.
Since $\mathcal{O}(W)$ contains a reflection of determinant $-1$, one can choose $h$ in such a way that $\rho_{\breve{\mathbb{Q}}_p} \circ h_{\breve{\mathbb{Q}}_p} \in G(\breve{\mathbb{Q}}_p)$. A different choice of $h$ amounts to changing the element in $G(\breve{\mathbb{Q}}_p)$ by an element of $\mathcal{G}(W)$, and therefore the corresponding element in $X_{\mu}^G(b)$ is independent of the choice of $h$.
Similarly, the composition $\rho_{\mathbb{A}_f^p} \circ \eta_{\mathbb{A}_f^p}$ gives an isometry
$$ \Lambda_{\mathbb{A}_f^p} \overset{\eta_{\mathbb{A}_f^p}}{\longrightarrow} H_{\acute{e}t}^2(X, \mathbb{A}_f^p) \cong NS(X)_{\mathbb{A}_f^p} \overset{\rho_{\mathbb{A}_f^p}}{\longrightarrow} N_{\mathbb{A}_f^p} \cong \Lambda_{\mathbb{A}_f^p} $$
preserving $\lambda$.
By our assumption that $n$ is the power of a prime number, the group $\mathcal{O}_n(\mathbb{Z}_{\ell})$ contains a reflection of determinant $-1$ for all but one $\ell$. Replacing $\eta$ with a different element of the coset $\eta \mathcal{O}_n(\hat{\mathbb{Z}}^p)$ therefore allows us to change the determinant at all places except for the prime dividing $n$. This shows that we may choose $\eta$ and $\rho$ in such a way that $\rho_{\mathbb{A}_f^p} \circ \eta_{\mathbb{A}_f^p} \in G(\mathbb{A}_f^p)$.
We define the image of $(X, \xi, \eta \mathcal{O}_n(\hat{\mathbb{Z}}^p))$ to be the class of the pair $(\rho_{\breve{\mathbb{Q}}_p} \circ h_{\breve{\mathbb{Q}}_p}, \rho_{\mathbb{A}_f^p} \circ \eta_{\mathbb{A}_f^p})$ in the double quotient $$I_{\phi}(\mathbb{Q}) \backslash X_{\mu}^G(b) \times G(\mathbb{A}_f^p)/K_n^p.$$

We claim that this construction is independent of all choices. Namely, a different isogeny $\rho': NS(X)_{\mathbb{Q}} \to N_{\mathbb{Q}}$ is of the form $\rho' = q\circ\rho$ for $q \in O(N_{d,\mathbb{Q}})$. Similarly, a different choice of $\eta$ amounts to replacing it by $\eta' = \eta \circ k$ for some $k \in \mathcal{O}_n(\hat{\mathbb{Z}}^p)$.
Recall that we restricted ourselves to choices such that $\rho'_{\mathbb{A}_f^p} \circ \eta'_{\mathbb{A}_f^p} \in G(\mathbb{A}_f^p)$.
As a consequence, $\mathrm{det}(q) \cdot \mathrm{det}(k)=1$. Since $n \ge 3$, the element $k \in \mathcal{O}_n(\hat{\mathbb{Z}}^p)$ has a component $k_{\ell} \in \mathcal{O}_n(\mathbb{Z}_{\ell})$ of determinant $1$. It follows that $\mathrm{det}(q) = \mathrm{det}(k)=1$.
We conclude that $q \in I_{\phi}(\mathbb{Q}) = SO(N_{d, \mathbb{Q}})$ and $k \in K_n^p$, and hence the corresponding element in the double quotient $$I_{\phi}(\mathbb{Q}) \backslash X_{\mu}^G(b) \times G(\mathbb{A}_f^p)/K_n^p$$ is independent of all choices.

\begin{rem} \textnormal{ }
\begin{enumerate}[(i)]
\item
We only study the period map on the perfection of the reduction, since in this case the target of the period map is a Deligne-Lusztig variety, which has a nice group-theoretic description.
\item
Strictly speaking, we did not show that the class of $\rho_{\breve{\mathbb{Q}}_p} \circ h_{\breve{\mathbb{Q}}_p} \in G(\breve{\mathbb{Q}}_p)$ lies in $X_{\mu}^G(b)$. This can be seen either from the classical fact that the Hodge polygon of the crystal $P^2_{cris}(X/W)(1)$ coincides with the Hodge polygon of $\mu$ (Mazur, Ogus, Nygaard; cf. the survey \cite{Liedtke}, Theorem 3.8), or from the description below which shows that $\rho_{\breve{\mathbb{Q}}_p} \circ h_{\breve{\mathbb{Q}}_p}$ arises from an element in $X_{\tilde{\mu}}^{\tilde{G}}(\tilde{b})$ via the map $\tilde{G}(\breve{\mathbb{Q}}_p) \to G(\breve{\mathbb{Q}}_p)$.
\end{enumerate}
\end{rem}

We sketch how to justify that the morphism defined in this direct manner coincides with the construction in (\ref{perconstr}) using the integral period map and Rapoport-Zink uniformi\-zation.
\begin{claim} The period map constructed in (\ref{perconstr}) coincides with the above direct description.
\end{claim}
\begin{proof}
We recall that we have a commutative diagram
$$\xymatrix{ & \widetilde{\mathscr{S}_{n,} }_{\bar{\mathbb{F}}_p}^{\tilde{b}, perf} \ar[r]^-{\sim}\ar[d] & I_{\tilde{\phi}}(\mathbb{Q})\backslash X_{\tilde{\mu}}^{\tilde{G}}(\tilde{b}) \times \tilde{G}(\mathbb{A}_f^p)/\tilde{K}_n^p \ar[d] \\
\overline{M}_{2d,ss,n}^{perf} \ar@{^{(}->}[r] & \overline{\mathscr{S}}_n^{b,perf} \ar[r]^-{\sim}& I_{\phi}(\mathbb{Q})\backslash X_{\mu}^G(b) \times G(\mathbb{A}_f^p)/K_n^p.
} $$

Lift the $\bar{\mathbb{F}}_p$-point of $\overline{\mathscr{S}}_n^{b}$ corresponding to the auxiliary polarized K3 surface $(\mathbb{X}, \mathbb{L})$ to an $\bar{\mathbb{F}}_p$-point of $\widetilde{\mathscr{S}_{n,} }_{\bar{\mathbb{F}}_p}^{\tilde{b}} $, and denote by $\mathbb{A}$ the associated abelian variety with distinguished subspaces $L_{\ell, \mathbb{A}} \subset \mathrm{End}(H_{\acute{e}t}^1(\mathbb{A},\mathbb{Z}_{\ell}))$ and $L_{cris, \mathbb{A}} \subset \mathrm{End}(H_{cris}^1(\mathbb{A}/W))$.
Let $y \in \overline{M}_{2d,ss,n}(\bar{\mathbb{F}}_p)$ correspond to a polarized K3 surface $ (X, \xi, \eta \mathcal{O}_n(\hat{\mathbb{Z}}^p) )$ with level structure. A lift $\widetilde{\varphi(y)} \in \widetilde{\mathscr{S}_{n,} }_{\bar{\mathbb{F}}_p}^{\tilde{b}}(\bar{\mathbb{F}}_p)$ of the image $\varphi(y) \in \overline{\mathscr{S}}_n^{b}(\bar{\mathbb{F}}_p)$ gives a Kuga-Satake abelian variety $A$ with level structure $\tilde{\eta}\tilde{K}_n^p$ and distinguished subspaces $L_{\ell, A} \subset \mathrm{End}(H_{\acute{e}t}^1(A,\mathbb{Z}_{\ell})) $ and $L_{cris, A} \subset \mathrm{End}(H_{cris}^1(A/W))$.

From the description of the Rapoport-Zink uniformization in the Hodge type case, the image of $\widetilde{\varphi(y)}$ in $$I_{\tilde{\phi}}(\mathbb{Q})\backslash X_{\tilde{\mu}}^{\tilde{G}}(\tilde{b}) \times \tilde{G}(\mathbb{A}_f^p)/\tilde{K}_n^p$$ is constructed as follows:
one first chooses a $GSpin$-isogeny $\tilde{\rho} \in \mathrm{Hom}(\mathbb{A},A) \otimes \mathbb{Q}$ in the sense of (\cite{Isogenies}, Definition 3.3). By the definition of a $GSpin$-isogeny, the realizations $\tilde{\rho}_{\mathbb{Q}_{\ell}} \in \mathrm{Isom}(H_{\acute{e}t}^1(A,\mathbb{Q}_{\ell}), H_{\acute{e}t}^1(\mathbb{A},\mathbb{Q}_{\ell}))$ and $\tilde{\rho}_{\breve{\mathbb{Q}}_p} \in \mathrm{Isom}(H_{cris}^1(A/W)[\frac{1}{p}], H_{cris}^1(\mathbb{A}/W)[\frac{1}{p}])$ match up the subspaces $L_{\ell, A}\otimes_{\mathbb{Z}_{\ell}} \mathbb{Q}_{\ell}$ and $L_{\ell, \mathbb{A}} \otimes_{\mathbb{Z}_{\ell}} \mathbb{Q}_{\ell}$, respectively $L_{cris, A}[\frac{1}{p}]$ and $L_{cris, \mathbb{A}}[\frac{1}{p}]$.
It follows from Theorem \ref{Madapusi-Pera}(\ref{NSKS}) that $\tilde{\rho}$ induces an isometry $\rho: NS(X)_{\mathbb{Q}} \overset{\sim}{\to} N_{\mathbb{Q}}$ mapping $\xi$ to $\lambda$ (compare also \cite{Isogenies}, Proposition 5.2).

We obtain an isometry
$$CL(\Lambda_{d,\mathbb{A}_f^p}) \overset{\tilde{\eta}_{\mathbb{A}_f^p}}{\longrightarrow} H_{\acute{e}t}^1(A, \mathbb{A}_f^p) \overset{\tilde{\rho}_{\mathbb{A}_f^p}}{\longrightarrow} H_{\acute{e}t}^1(\mathbb{A}, \mathbb{A}_f^p) \cong CL(\Lambda_{d,\mathbb{A}_f^p}) $$
which is an element of $\tilde{G}(\mathbb{A}_f^p)$ and in particular preserves the subspace $\Lambda_{d, \mathbb{A}_f^p}$.
Similarly, after choosing an appropriate $\tilde{h}: CL(\Lambda_{d,W}) \cong H^1_{cris}(A/W)$ matching up the subspaces $\Lambda_{d,W}$ and $L_{cris, A}$, we get
$$CL(\Lambda_{d,\breve{\mathbb{Q}}_p}) \overset{\tilde{h}_{\breve{\mathbb{Q}}_p}}{\longrightarrow} H^1_{cris}(A/W)[\frac{1}{p}] \overset{\tilde{\rho}_{\breve{\mathbb{Q}}_p}}{\longrightarrow} H^1_{cris}(\mathbb{A}/W)[\frac{1}{p}] \cong CL(\Lambda_{d,\breve{\mathbb{Q}}_p})$$
which we can view as an element of $\tilde{G}(\breve{\mathbb{Q}}_p)$.
Then $\tilde{h}$ induces an isometry $h:\Lambda_W \overset{\sim}{\to} H^2_{cris}(X/W)$ mapping $\lambda$ to $c_{cris}(\xi)$ using Theorem \ref{Madapusi-Pera}(\ref{naturalsubspaces}).
The image of $\widetilde{\varphi(y)} $ in $$I_{\tilde{\phi}}(\mathbb{Q})\backslash X_{\tilde{\mu}}^{\tilde{G}}(\tilde{b}) \times \tilde{G}(\mathbb{A}_f^p)/\tilde{K}_n^p$$ is then the class of the pair $(\tilde{\rho}_{\breve{\mathbb{Q}}_p} \circ \tilde{h}_{\breve{\mathbb{Q}}_p}, \tilde{\rho}_{\mathbb{A}_f^p} \circ \tilde{\eta}_{\mathbb{A}_f^p})$.
Via the natural map $\tilde{G} \to G$ and using Theorem \ref{Madapusi-Pera} it follows that the image in $I_{\phi}(\mathbb{Q})\backslash X_{\mu}^G(b) \times G(\mathbb{A}_f^p)/K_n^p$ is precisely the class of the pair $(\rho_{\breve{\mathbb{Q}}_p} \circ h_{\breve{\mathbb{Q}}_p}, \rho_{\mathbb{A}_f^p} \circ \eta_{\mathbb{A}_f^p})$, as desired.
\end{proof}

The description on $\bar{\mathbb{F}}_p$-points translates the injectivity of the period map into the following statement:

\begin{thm}[\cite{Ogus}, Theorem II]
Let $(X, \xi)$ and $(X', \xi')$ be two polarized K3 surfaces over $\bar{\mathbb{F}}_p$ and $\iota: NS(X') \overset{\sim}{\rightarrow} NS(X)$ an isometry mapping $\xi'$ to $\xi$, which extends to a diagram
$$ \xymatrix{ NS(X') \ar[r]^{\iota}\ar[d] & NS(X) \ar[d] \\
H^2_{cris}(X'/W) \ar[r]^{\sim} & H^2_{cris}(X/W), }$$
where the lower map is an isomorphism of $F$-crystals.
Then $\iota $ is induced by a unique isomorphism of K3 surfaces $f: X \overset{\sim}{\to} X'$.
\end{thm}

The injectivity of the $p$-adic period map can thus be interpreted as a geometric version of the crystalline Torelli theorem of Ogus.

\section{The image of the period map}

We recall the computation of the image of the complex period map for polarized K3 surfaces as in (\cite{K3global}, Remark 6.3.7) and compute the image of the $p$-adic period map in a similar fashion.

\subsection{The image of the complex period map}

We remind the reader of the computation of the image of the complex period map for polarized K3 surfaces.
Relevant references are \cite{Kulikov} and \cite{K3global}.
Recall that in Section \ref{modspace} we defined a moduli space $M^*_{2d,n,\mathbb{C}}$ of quasi-polarized K3 surfaces of degree $2d$ with level-$n$ structure.
There is a natural extension $$\Phi^*: M^*_{2d,n,\mathbb{C}} \to \mathcal{G}_n(\mathbb{Z}) \backslash \mathcal{D}$$ of the complex period map $\Phi$ to $M^*_{2d,n,\mathbb{C}}$, which is shown to be surjective in \cite{Kulikov}. Note however that contrary to $\Phi$, the extended period map $\Phi^*$ is not an open immersion anymore.
Kulikov`s result shows that in order to compute the image of $M_{2d,n,\mathbb{C}}$
it is enough to establish a Hodge theoretic criterion to distinguish between polarizations and quasi-polarizations of complex K3 surfaces, and use it to exhibit which points of $\mathcal{G}_n(\mathbb{Z})\backslash \mathcal{D}$ correspond to polarized as opposed to merely quasi-polarized K3 surfaces.

The starting point will be the following Proposition from \cite{K3global}.
\begin{prop}\label{criterionampleness}
Let $(X, \xi)$ be a quasi-polarized K3 surface over an algebraically closed field $F$. Then $\xi$ is ample (and hence a polarization) if and only if $\langle \xi, \gamma\rangle \not=0$ for any class $\gamma \in NS(X)$ with $\gamma^2=-2$.
\end{prop}
\begin{proof}

Suppose that $\langle \xi, \gamma \rangle=0$ for a class with $\gamma^2=-2$. By Riemann-Roch, $\gamma$ is of the form $\pm [C]$ for a $(-2)$-curve $C$ on $X$. Since $\langle [C], \xi \rangle = 0$, this shows that $\xi$ cannot be ample.
Conversely, if $\xi$ is a quasi-polarization, i.e. big and nef, then $\xi$ is ample unless there exists a $(-2)$-curve $C$ on $X$ with $\langle [C], \xi \rangle =0$, see (\cite{K3global}, Corollary 8.1.7).
\end{proof}

For a complex K3 surface $X$ the Lefschetz (1,1)-theorem asserts that $$NS(X) = H_B^2(X,\mathbb{Z}) \cap H^{1,1}(X) = H_B^2(X,\mathbb{Z}) \cap H^{2,0}(X)^{\perp}. $$
In view of Proposition \ref{criterionampleness} this gives a purely Hodge theoretic way to distinguish between polarized and quasi-polarized K3 surfaces: a quasi-polarization $\xi$ is a polarization if and only if there does not exist a $\gamma \in H^2(X, \mathbb{Z})$ of square $-2$ such that $\langle \gamma, c_B(\xi) \rangle=0$ and $ H^{2,0}(X) \perp \gamma$.
Given an isometry $\iota: \Lambda \overset{\sim}{\rightarrow} H_B^2(X, \mathbb{Z})$ mapping $\lambda$ to $c_B(\xi)$, it follows that $\xi$ is a polarization if and only if the corresponding point in $\mathcal{D}$ does not lie in $\delta^{\perp} \subset \mathbb{P}(\Lambda_{d, \mathbb{C}})$ for any $\delta \in \Delta(\Lambda_d)$,
where $\Delta(\Lambda_d) = \{ \delta \in \Lambda_d \, | \, \delta^2=-2 \}$ denotes the set of roots of the lattice $\Lambda_d$.

Define $\mathcal{D}^{\circ} \subset \mathcal{D}$ as the complement of $\bigcup_{\delta \in \Delta(\Lambda_d)} \delta^{\perp}$.
The above discussion shows the following uniformization of the moduli space of polarized K3 surfaces.
\begin{thm}[\cite{K3global}, Remark 6.3.7] \label{imagecomplex}
The period map $\Phi$ gives a uniformization $$ M_{2d,n,\mathbb{C}} \overset{\sim}{\rightarrow} \mathcal{G}_n(\mathbb{Z}) \backslash \mathcal{D}^{\circ}.$$
\end{thm}

We remark that this gives an explicit description of the moduli space $M_{2d,n,\mathbb{C}} $ of polarized complex K3 surfaces of degree $2d$ purely in terms of linear algebra.
Note that $\mathcal{D}^{\circ}$ can be viewed as a moduli space of polarized K3 surfaces $(X, \xi)$ together with a marking $\iota: \Lambda \overset{\sim}{\rightarrow} H_B^2(X, \mathbb{Z})$ mapping $\lambda$ to $c_B(\xi)$.

\subsection{The image of the $p$-adic period map}

In this section we compute the image of the $p$-adic period map

\begin{equation}\label{periodmap2} \overline{\varphi}: \overline{M}_{2d,ss,n}^{perf} \hookrightarrow I_{\phi}(\mathbb{Q}) \backslash X_{\mu}^G(b) \times G(\mathbb{A}_f^p)/K_n^p. \end{equation}

\begin{mydef} The set of roots of $N_{d, \mathbb{Q}}$ is the set $\Delta(N_{d, \mathbb{Q}}):= \{\delta \in N_{d, \mathbb{Q}} | \delta^2 = -2 \}$.
\end{mydef}
The action of $I_{\phi}(\mathbb{Q})$ on $N_{d, \mathbb{Q}}$ preserves the subset $\Delta(N_{d, \mathbb{Q}})$.

For every $\delta \in \Delta(N_{d, \mathbb{Q}})$ we
define a subset
$C_{\delta}$ of $X_{\mu}^G(b)(\bar{\mathbb{F}}_p) \times G(\mathbb{A}_f^p)$ as the set of pairs $(x,g)$ such that $x^{-1}(\delta) \in \Lambda_{d,W}$ and $g^{-1}(\delta) \in \Lambda_{d, \hat{\mathbb{Z}}^p}$.
Here we view $\delta \in N_{d, \mathbb{Q}}$ as an element of $\Lambda_{d, \mathbb{A}_f^p}$ and $\Lambda_{\breve{\mathbb{Q}}_p}$ via the identifications in (\ref{padicidentification}).

Clearly, this set is stable under the action of $K_n^p$.
Under the left action of $q \in I_{\phi}(\mathbb{Q})$, the subset $C_{\delta}$ is mapped to $C_{q^{-1}\delta}$.
Hence we can look at the double quotient \begin{equation} \label{C} C:= I_{\phi}(\mathbb{Q}) \backslash \bigcup_{\delta \in \Delta(N_{d, \mathbb{Q}})} C_{\delta} /K_n^p, \end{equation} which as we will now see is the set of $\bar{\mathbb{F}}_p$-points of a closed subscheme of $$I_{\phi}(\mathbb{Q}) \backslash X_{\mu}^G(b) \times G(\mathbb{A}_f^p)/K_n^p.$$

\begin{thm}\label{impadicper}
Suppose $p > 18d+4$.
The set of $\bar{\mathbb{F}}_p$-points of the image of the period map (\ref{periodmap2}) is the complement of the subset $C$.
\end{thm}
\begin{proof}
By (\cite{Matsumoto}, Theorem 4.1) the period map
$$ \overline{\mathcal{M}}^{*, perf}_{2d,ss,n} \to I_{\phi}(\mathbb{Q}) \backslash X_{\mu}^G(b) \times G(\mathbb{A}_f^p)/K_n^p $$ is surjective provided that $p >18d+4$.
Consequently, given a point $[x,g]$ of the right hand side, there exists a quasi-polarized K3 surface $(X, \xi, \eta \mathcal{O}_n(\hat{\mathbb{Z}}^p))$ over $\bar{\mathbb{F}}_p$ mapping to $[x,g]$. We need to show that $\xi$ is in fact a polarization if and only if $[x,g] \notin C$.

By construction, after choosing suitable isometries $\rho: NS(X)_{\mathbb{Q}} \overset{\sim}{\to} N_{\mathbb{Q}}$ and $h: \Lambda_W \overset{\sim}{\to} H^2_{cris}(X/W)$, the pair $[x,g]$ is given by $x= \rho_{\breve{\mathbb{Q}}_p} \circ h_{\breve{\mathbb{Q}}_p}$ and $g=\rho_{\mathbb{A}_f^p} \circ \eta_{\mathbb{A}_f^p} $.
\begin{claim}
With this choice of $\rho$, the pair $(x,g) \in X_{\mu}^G(b)(\bar{\mathbb{F}}_p) \times G(\mathbb{A}_f^p) $ lies in $C_{\delta}$ if and only if $\rho^{-1}(\delta) \in NS(X)$.
\end{claim}
\begin{proof}[Proof of Claim]
It is obvious that $\rho^{-1}(\delta) \in NS(X)_{\mathbb{Q}}$.
The pair $(x,g)$ belongs to $C_{\delta}$ if and only if the defining conditions $x^{-1}(\delta) \in \Lambda_{d,W}$ and $g^{-1}(\delta) \in \Lambda_{d,\hat{\mathbb{Z}}^p}$ hold, which translate to
$$ \rho_{\breve{\mathbb{Q}}_p}^{-1}(\delta) \in P_{cris}^2(X/W) \subset P_{cris}^2(X/W)[\frac{1}{p}],$$
$$ \rho_{\mathbb{A}_f^p}^{-1}(\delta) \in P_{\acute{e}t}^2(X, \hat{\mathbb{Z}}^p) \subset P_{\acute{e}t}^2(X, \mathbb{A}_f^p).$$
It follows from Proposition \ref{NSintegral} that this is the case if and only if $\rho^{-1}(\delta) \in NS(X)$.
\end{proof}
Since $\gamma= \rho^{-1}(\delta) \in NS(X)_{\mathbb{Q}}$ is a class with $\gamma^2=-2$ and $\langle \gamma, \xi \rangle =0$, the Theorem follows from Proposition \ref{criterionampleness}.
\end{proof}

Recall that $$I_{\phi}(\mathbb{Q}) \backslash X_{\mu}^G(b) \times G(\mathbb{A}_f^p)/K_n^p = \coprod_g \Gamma_g \backslash X_{\mu}^G(b),$$
where $g$ runs through a set of representatives of the double quotient $I_{\phi}(\mathbb{Q}) \backslash G(\mathbb{A}_f^p)/K_n^p$.

We reformulate the description of the image of the $p$-adic period map as an open subset of this disjoint union.

\begin{mydef}\label{defZdelta}
For $\delta \in \Delta(N_{d, \mathbb{Q}})$, let $Z_{\delta} \subset X_{\mu}^G(b)(\bar{\mathbb{F}}_p)$ be the subset consisting of those $ x \in G(\breve{\mathbb{Q}}_p)$ such that $x^{-1}(\delta) \in \Lambda_{d,W}$.
\end{mydef}

Denote by $$N_{d,g} := N_{d, \mathbb{Q}} \cap g(\Lambda_{d,\hat{\mathbb{Z}}^p}),$$ a $\mathbb{Z}[\frac{1}{p}]$-lattice which does not depend on the choice of the representative $g$ of an element of $I_{\phi}(\mathbb{Q}) \backslash G(\mathbb{A}_f^p)/K_n^p$.
Define $X^{\circ}_g(\bar{\mathbb{F}}_p)$ to be the subset $$X^{\circ}_g(\bar{\mathbb{F}}_p) = X_{\mu}^G(b)(\bar{\mathbb{F}}_p) \setminus \bigcup_{\delta \in \Delta(N_{d,g})} Z_{\delta},$$
where $\Delta(N_{d,g}) := \Delta(N_{d, \mathbb{Q}}) \cap N_{d,g}$.

We can identify the complement of $C$ with the disjoint union $$ \coprod_g \Gamma_g \backslash X^{\circ}_g(\bar{\mathbb{F}}_p).$$
Note that Theorem \ref{impadicper} shows that each term in this disjoint union is the set of $\bar{\mathbb{F}}_p$-points of an open subscheme $\Gamma_g \backslash X^{\circ}_g \subset \Gamma_g \backslash X_{\mu}^G(b) $, and therefore also $X^{\circ}_g \subset X_{\mu}^G(b) $ is an open subscheme.

\begin{cor}\label{unifordisj}
Suppose $p > 18d+4$.
The period map induces an isomorphism
$$\overline{M}_{2d,ss,n}^{perf} \cong \coprod_g \Gamma_g \backslash X^{\circ}_g.$$
\end{cor}

\begin{rem}
The above discussion shows that $X^{\circ}_g(\bar{\mathbb{F}}_p)$ can be naturally identified with the set of polarized K3 surfaces $(X, \xi)$ of degree $2d$ over $\bar{\mathbb{F}}_p$ together with an isometry $\iota: NS(X)[\frac{1}{p}] \overset{\sim}{\rightarrow} N_{g}$ mapping $\xi$ to $\lambda$, where $N_g = N_{\mathbb{Q}} \cap g(\Lambda_{\hat{\mathbb{Z}}^p})$.
\end{rem}

\subsection{$p$-adic uniformization of the K3 moduli space}

In this section we come to the rigid analytic uniformization of the generic fiber of $\breve{M}_{2d,ss,n,W}$.
This amounts to computing the image of the period map \begin{equation}\label{padicperiodmap} M_{2d,n, \breve{\mathbb{Q}}_p}^{ss}:=(\breve{M}_{2d,ss,n,W})_{\eta}^{rig} \hookrightarrow \coprod_g \Gamma_g \backslash \mathcal{M}. \end{equation}

Recall that there is a continuous specialization map
$sp: \left|\mathcal{M}\right| \to \left|X_{\mu}^G(b)\right|$ from the local Shimura variety to the affine Deligne-Lusztig variety.
We define $\mathcal{M}_g^{\circ} := sp^{-1}(X^{\circ}_g)$ to be the tube over the open subset $X^{\circ}_g$, which is a quasi-compact open rigid subvariety of $\mathcal{M}$.

\begin{thm}\label{padicuniformization}
Suppose $p > 18d+4$.
The period map (\ref{padicperiodmap}) induces a uniformization
\begin{equation}\label{equationpadicuniformization} M_{2d,n, \breve{\mathbb{Q}}_p}^{ss} \cong \coprod_g \Gamma_g \backslash \mathcal{M}^{\circ}_g \end{equation}
of rigid spaces over $\breve{\mathbb{Q}}_p$.
\end{thm}
\begin{proof}
We also denote by $sp: \left| Sh^b\right| \to \left| \overline{\mathscr{S}}_n^b\right|$ be the specialization map for the formal scheme $\widehat{ {\mathscr{S}_{n,W}}_{/\overline{\mathscr{S}}_n^b}}$ over $\mathrm{Spf \,} W$.
Since $\breve{M}_{2d,ss,n,W}$ is a formal open subscheme, we have
$$ (\breve{M}_{2d,ss,n,W})_{\eta}^{rig} = sp^{-1}(\left|\overline{M}_{2d,ss,n}\right|).$$
The Rapoport-Zink uniformization of the basic stratum in the special fiber of the Shimura variety $\mathscr{S}_{n,W}$ induces an identification of topological spaces
$$ \left| \overline{\mathscr{S}}_n^b\right| = \left| \coprod_g \Gamma_g \backslash X_{\mu}^G(b) \right|$$
and Corollary \ref{unifordisj} shows that the period map identifies the open subspaces $$\left| \overline{M}_{2d,ss,n}\right| = \left| \coprod_g \Gamma_g \backslash X^{\circ}_g \right|.$$
We conclude that
$$ (\breve{M}_{2d,ss,n,W})_{\eta}^{rig} = sp^{-1}(\left| \coprod_g \Gamma_g \backslash X^{\circ}_g \right|) = \coprod_g \Gamma_g \backslash \mathcal{M}^{\circ}_g.$$
\end{proof}

\begin{rem}
Without giving an explicit description of the uniformizing space $\mathcal{N}$, it is claimed in (\cite{Xu}, Corollary 8.12) that there is a uniformization
$$ M_{2d,n, \breve{\mathbb{Q}}_p}^{ss} \cong \coprod_g \Gamma_g \backslash \mathcal{N}.$$
While one could expect the existence of such a uniformization from the case of Shimura varieties, it turns out to be the case that for the moduli space of polarized K3 surfaces, in the disjoint union (\ref{equationpadicuniformization}) the uniformizing domain $\mathcal{M}^{\circ}_g$ varies with $g$.
Furthermore, it is claimed in (\cite{Xu}, 8.5) that $\mathcal{N}$ is stable under the action of the $p$-adic Lie group $J_b(\mathbb{Q}_p)$. Instead, we actually see that the spaces $\mathcal{M}_{g}^{\circ}$ are only stable under the action of the discrete subgroup $\Gamma_g \subset J_b(\mathbb{Q}_p)$.
This is analogous to the observation that the complex uniformizing domain $\mathcal{D}^{\circ}$ is only stable under the discrete group $\mathcal{G}(\mathbb{Z}) \subset G(\mathbb{R})$, and not under the full real Lie group $G(\mathbb{R})$.
\end{rem}

\subsection{Special subvarieties and the image of the period map}\label{spimage}

We give a slightly different viewpoint on the results of the previous sections by interpre\-ting the image of the K3 period map as the complement of a union of special subvarieties. This is in line with (\cite{Achter}, Remark 5.4).

For $\delta \in \Delta(\Lambda_d)$, denote by $G_{\delta} = \{g \in G \, | \, g \delta = \delta\}$ the subgroup of $G$ fixing $\delta$.
Then $(G_{\delta}, \delta^{\perp}\cap \mathcal{D})$ is a sub-Shimura datum and we denote by \begin{equation}\label{ci} Sh_{\delta,n, \mathbb{Q}} \to Sh_{n, \mathbb{Q}}\end{equation} the induced morphism of Shimura varieties. For $n$ large enough, this morphism is a closed immersion, see (\cite{DeligneShimura}, Proposition 1.15), and thus we can interpret $Sh_{\delta,n, \mathbb{Q}}$ as a closed subvariety of $Sh_{n ,\mathbb{Q}}$ which only depends on the class of $\delta$ in $\Delta(\Lambda_d) / \mathcal{G}_n(\mathbb{Z})$.
Theorem \ref{imagecomplex} shows that the image of the K3 period map is precisely the complement of the union of the closed subvarieties $Sh_{\delta,n,\mathbb{Q}}$ as $\delta$ runs through a set of representatives for the (finite) set $\Delta(\Lambda_d) / \mathcal{G}_n(\mathbb{Z})$.

We want to use the theory of integral canonical models of Shimura varieties to extend this picture over $\mathbb{Z}[\frac{1}{N}]$ for $N$ large enough.
Namely, for almost all primes $p$, the group $K_{\delta,p}:= \mathcal{G}(\mathbb{Z}_p) \cap G_{\delta}(\mathbb{Q}_p)$ is a hyperspecial subgroup, so that there is an integral canonical model $\mathscr{S}_{\delta,n,\mathbb{Z}[\frac{1}{N}]}$ such that the closed immersion (\ref{ci}) extends to a closed immersion
$$\mathscr{S}_{\delta,n,\mathbb{Z}[\frac{1}{N}]} \hookrightarrow \mathscr{S}_{n, \mathbb{Z}[\frac{1}{N}]}$$ for $N$ sufficiently large.
\begin{prop}
For sufficiently large $N$, the image of the open immersion $\varphi: M_{2d,n, \mathbb{Z}[\frac{1}{N}]} \hookrightarrow \mathscr{S}_{n, \mathbb{Z}[\frac{1}{N}]}$ is the complement of the union of the closed subschemes $\mathscr{S}_{\delta,n,\mathbb{Z}[\frac{1}{N}]}$ as $\delta $ runs through the finite set $\Delta(\Lambda_d) / \mathcal{G}_{n}(\mathbb{Z})$.
\end{prop}
\begin{proof}
For $N$ sufficiently large, both $\mathscr{S}_{n, \mathbb{Z}[\frac{1}{N}]} \setminus M_{2d,n, \mathbb{Z}[\frac{1}{N}]}$ and $\bigcup_{\delta} \mathscr{S}_{\delta,n,\mathbb{Z}[\frac{1}{N}]} $ are closed subschemes which are flat over $\mathbb{Z}[\frac{1}{N}]$.
The proposition then follows from the fact that their generic fibers can be identified using Theorem \ref{imagecomplex}.
\end{proof}

Rapoport-Zink uniformization for the Shimura variety $\mathscr{S}_{\delta,n,\mathbb{Z}[\frac{1}{N}]}$ gives rise to a commutative diagram
\begin{equation} \label{RZdiagram} \xymatrix{ \overline{\mathscr{S}}_{\delta,n}^{b,perf} \ar@{^{(}->}[r] \ar[d]^{\cong} & \overline{\mathscr{S}}_{n}^{b,perf} \ar[d]^{\cong}\\
I_{\phi_{\delta}}(\mathbb{Q})\backslash X_{\delta} \times G_{\delta}(\mathbb{A}_f^p)/K_{\delta}^p \ar@{^{(}->}[r] & I_{\phi}(\mathbb{Q})\backslash X \times G(\mathbb{A}_f^p)/K_n^p.
}\end{equation}

\begin{claim}
The union of the images of the lower map in (\ref{RZdiagram}) as $\delta$ runs through $\Delta(\Lambda_d) / \mathcal{G}_n(\mathbb{Z})$ has as $\bar{\mathbb{F}}_p$-points precisely the closed subset $$ C:= I_{\phi}(\mathbb{Q}) \backslash \bigcup_{\overline{\delta} \in \Delta(N_{d, \mathbb{Q}})} C_{\overline{\delta}} /K_n^p$$
as defined in (\ref{C}).
\end{claim}
\begin{proof}[Proof of Claim]
The element $\delta \in \Lambda_d$ gives rise to a generic Hodge cycle over $Sh_{\delta,n,\mathbb{C}}$, i.e. a global section $\delta \in H^0(Sh_{\delta,n, \mathbb{C}},\mathbb{V}_{| Sh_{\delta,n, \mathbb{C}}})$ which is a Hodge cycle at every point.
Denoting by $\widetilde{Sh}_{\delta,n, \mathbb{C}}$ the preimage of $Sh_{\delta,n, \mathbb{C}}$ in $\widetilde{Sh}_{n, \mathbb{C}}$,
the element $\delta$ corresponds to a generic Hodge cycle in $\tilde{\mathbb{V}}^{\otimes (1,1)}$, and thus to a section $\delta \in \underline{\mathrm{End}}(\mathfrak{A})(\widetilde{Sh}_{\delta,n, \mathbb{Q}})$ over $\widetilde{Sh}_{\delta,n, \mathbb{Q}} $ of the endomorphisms of the universal abelian scheme $f: \mathfrak{A} \to Sh_{n,\mathbb{Q}}$.
By a Néron model property, it extends to a section $\delta \in \underline{\mathrm{End}}(\mathcal{A})(\widetilde{\mathscr{S}}_{\delta,n})$ over $\widetilde{\mathscr{S}}_{\delta,n} $ of the endomorphisms of the universal abelian scheme $\mathcal{A} \to \widetilde{\mathscr{S}}_{n}$.
Here $\widetilde{\mathscr{S}}_{\delta,n}$ denotes the preimage of $\mathscr{S}_{\delta,n}$ in $\widetilde{\mathscr{S}}_{n}$.
Let $y \in \overline{\mathscr{S}}_{\delta,n}^{b}(\bar{\mathbb{F}}_p)$ with a lift $\tilde{y} \in \widetilde{\mathscr{S}_{n,} }_{\bar{\mathbb{F}}_p}^{\tilde{b}}(\bar{\mathbb{F}}_p)$ and denote by $\mathcal{A}_{\tilde{y}}$ the abelian variety over $\bar{\mathbb{F}}_p$ which is the fiber of the universal abelian scheme $\mathcal{A}$ over $\tilde{y}$. Then $\delta_{\tilde{y}} \in \mathrm{End}(\mathcal{A}_{\tilde{y}})$ is a special endomorphism, i.e. its cohomological realizations lie in the distinguished subspaces $L_{\ell,\tilde{y}} \subset \mathrm{End}(H_{\acute{e}t}^1(\mathcal{A}_{\tilde{y}}, \mathbb{Z}_{\ell}))$ respectively $L_{cris,\tilde{y}} \subset \mathrm{End}(H_{cris}^1(\mathcal{A}_{\tilde{y}}/W))$. Denoting by $\mathbb{A}$ a Kuga-Satake abelian variety of the auxiliary polarized supersingular K3 surface $(\mathbb{X}, \mathbb{L})$ as in Section \ref{defpadicper}, we can choose a $GSpin$-isogeny between $\mathbb{A}$ and $\mathcal{A}_{\tilde{y}}$.
Under the $GSpin$-isogeny, $\delta_{\tilde{y}}$ corresponds to an element $\bar{\delta} \in L(\mathbb{A})_{\mathbb{Q}} = N_{d, \mathbb{Q}}$.
The image of $X_{\delta} \times G_{\delta}(\mathbb{A}_f^p) \to X \times G(\mathbb{A}_f^p)$ is given by the set of pairs $(x,g)$ such that $x^{-1} (\overline{\delta}) \in \mathcal{G}(W).\delta \subset \Lambda_{d,W}$ and $g^{-1}(\overline{\delta}) = \delta \in \Lambda_{d, \hat{\mathbb{Z}}^p}$.
We conclude that this image lies in $C_{\overline{\delta}}$, as claimed.

Conversely, if $y \in \overline{\mathscr{S}}_{n}^{b}(\bar{\mathbb{F}}_p)$ is a point that lies in the image of $C_{\overline{\delta}}$ in $$I_{\phi}(\mathbb{Q}) \backslash X \times G(\mathbb{A}_f^p)/K_n^p,$$
there is a lift $\tilde{y} \in \widetilde{\mathscr{S}_{n,} }_{\bar{\mathbb{F}}_p}^{\tilde{b}}(\bar{\mathbb{F}}_p)$ and a $GSpin$-isogeny $\rho$ between $\mathcal{A}_{\tilde{y}}$ and $\mathbb{A}$ such that $\rho^{-1}(\overline{\delta})$ lies in $L(\mathcal{A}_{\tilde{y}}) \subset \mathrm{End}(\mathcal{A}_{\tilde{y}})$. It is shown in (\cite{MP2}, Corollary 8.15) that the deformation space of $y$ together with the special endomorphism $\rho^{-1}(\overline{\delta})$ admits a flat component, and thus there is a field $F$ of characteristic $0$ and a point $Y \in Sh_{n, \mathbb{Q}}(F)$ lifting $y$ such that the special endomorphism $\rho^{-1}(\overline{\delta})$ lifts to $\delta \in L(\mathcal{A}_{Y})$, where $\mathcal{A}_Y$ is the Kuga-Satake abelian variety at $Y$.
For any embedding $\sigma: F \hookrightarrow \mathbb{C}$, the special endomorphism $\delta$ gives rise to a Hodge class $\delta \in P_B^2(\mathcal{A}_{\sigma(Y)}, \mathbb{Z}) \cong \Lambda_{d}$ with $\delta^2=-2$, and by construction $\sigma(Y) \in Sh_{\delta, n, \mathbb{Q}}(\mathbb{C})$.
As $Sh_{\delta,n,\mathbb{Q}}$ is defined over $\mathbb{Q}$, we conclude that $Y \in Sh_{\delta,n,\mathbb{Q}}(F)$, and thus $y \in \overline{\mathscr{S}}_{\delta, n}^{b}(\bar{\mathbb{F}}_p)$, since $ \mathscr{S}_{\delta, n}$ is the closure of $Sh_{\delta,n,\mathbb{Q}}$ in $\mathscr{S}_{n}$ (for $N$ large enough).
\end{proof}

\section{Cubic fourfolds}

We briefly sketch how the same ideas can be applied to obtain a Rapoport-Zink type uniformization for the moduli space of smooth cubic fourfolds with supersingular reduction.

Similarly as above, one can introduce a moduli space $M_{cf, n, \mathbb{Z}[\frac{1}{2}]}$ with level structure, cf. (\cite{MP}, 5.13).
The lattice $\Lambda$ is now the integral cohomology lattice of a cubic fourfold, which is a self-dual lattice of signature $(21,2)$. Let $\lambda \in \Lambda$ denote the square of the class of a hyperplane section, so that $\lambda^2=3$, and define $\Lambda_0 := \langle \lambda \rangle^{\perp} \subset \Lambda$ to be the orthogonal complement. For a precise description of these lattices, we refer the reader to (\cite{Hassett}, Proposition 2.1.2).
For $n \ge 3$ there is a complex analytic period map
$$\Phi: M_{cf,n, \mathbb{C}} \to \mathcal{G}_n(\mathbb{Z})\backslash \mathcal{D},$$
where
$$\mathcal{D} = \{x \in \mathbb{P}(\Lambda_0 \otimes \mathbb{C}) | \langle x,x \rangle =0, \langle x,\overline{x}\rangle < 0 \} \subset \mathbb{P}(\Lambda_0 \otimes \mathbb{C})$$ and $\mathcal{G}_n(\mathbb{Z})$ is the subgroup of $\mathcal{G}(\mathbb{Z}) := \{ g \in SO(\Lambda) \, | \, g \lambda = \lambda\} $ of elements that reduce to the identity modulo $n$.
By Voisin's Torelli theorem for cubic fourfolds \cite{VoisinTorelli}, this period map is an open immersion.
Again, for $n\ge 3$ the target of the period map is a Shimura variety $Sh_{n, \mathbb{C}}$ attached to the group $G:= SO(\Lambda_{0, \mathbb{Q}}) $ of signature $(20,2)$. By an argument of André (cf. \cite{AndreHK}, §6) the period map descents to a morphism of algebraic varieties $\Phi_{\mathbb{Q}}: M_{cf,n,\mathbb{Q}} \to Sh_{n, \mathbb{Q}}$.

Furthermore, assuming that $n$ is a power of a prime number and using again the theory of integral canonical models, we get an integral extension
$$\varphi: \mathcal{M}_{cf,n,\mathbb{Z}_{(p)}} \hookrightarrow \mathscr{S}_{n, \mathbb{Z}_{(p)}}$$
which is an open immersion for $n$ large enough, cf. (\cite{MP}, 5.13). We assume everywhere that $p >3$ and $p \nmid n$.
Rapoport-Zink uniformization provides us with a $p$-adic period map for the tube over the supersingular locus
$$ M_{cf, n, \breve{\mathbb{Q}}_p}^{ss}:= (\breve{M}_{cf, ss,n,W})_{\eta}^{rig} \hookrightarrow I_{\phi}(\mathbb{Q})\backslash \mathcal{M} \times G(\mathbb{A}_f^p)/K_n^p $$
The arguments in the proof of Proposition \ref{isogenygroup} can be adapted, replacing Theorem \ref{Madapusi-Pera} by its corresponding analog (\cite{MP}, Theorem 5.14(2)), to show that $I_{\phi} = SO(CH^2_{0,\mathbb{Q}})$.
Here $\mathbb{X}$ is a supersingular cubic fourfold over $\bar{\mathbb{F}}_p$, and $$CH^2_{0, \mathbb{Q}} := \langle h^2\rangle^{\perp} \subset CH^2_{\mathbb{Q}}:=CH^2(\mathbb{X})_{\mathbb{Q}}$$ is the complement of the square of a hyperplane section.

\subsection{The image of the period map: special cubic fourfolds}

As conjectured by Hassett in (\cite{Hassett}, §4.3) and proved by Laza and Loojenga in \cite{Laza} and \cite{Looijenga}, the image of the period map can be explicitly described as the complement of an arithmetic arrangement of hyperplanes, corresponding to the Hodge structures of so-called special cubic fourfolds of discriminant $2$ and $6$.

To define special cubic fourfolds, we start with a saturated, positive definite sub-lattice $K \subset \Lambda$ of rank two containing $\lambda$. We will write $K_0:= K\cap \Lambda_0$.
We let $$\mathcal{D}_K = \{ x \in \mathcal{D} \, | \, K_0 \subset x^{\perp}\} \subset \mathcal{D},$$
this is a hyperplane in $\mathcal{D}$, cf. (\cite{Hassett}, §3.1).
Denote by $\mathcal{C}_K$ the image of $\mathcal{D}_K$ in $\mathcal{G}_n(\mathbb{Z}) \backslash \mathcal{D}$, which clearly only depends on the orbit of $K$ under $\mathcal{G}_n(\mathbb{Z})$. Hassett shows in (\cite{Hassett}, Proposition 3.2.4) that the orbit of $K$ (and thus also $\mathcal{C}_K$) only depends on the discriminant of the lattice $K$.
We may thus write $\mathcal{C}_d$ for the divisor $\mathcal{C}_K$ with $K \subset \Lambda$ any saturated positive definite rank two lattice of discriminant $d$.
We denote by $\mathcal{D}_d$ the union of all $\mathcal{D}_K$, where $K$ runs through the saturated positive definite rank two lattices of discriminant $d$, and thus $\mathcal{C}_d = \mathcal{G}_n(\mathbb{Z}) \backslash \mathcal{D}_d $.

Set $\mathcal{D}^{\circ}:= \mathcal{D} \setminus (\mathcal{D}_2 \cup \mathcal{D}_6)$.
Then the result on the image of the period map for cubic fourfolds reads as follows:

\begin{thm}[\cite{Laza}, Theorem 1.1; \cite{Looijenga}]
The period map induces an isomorphism $$M_{cf, n, \mathbb{C}} \overset{\sim}{\to} \mathcal{G}_n(\mathbb{Z}) \backslash \mathcal{D}^{\circ}.$$
\end{thm}

In other words, as predicted by Hassett, the period map misses precisely the hyperplane arrangements $\mathcal{D}_2$ and $\mathcal{D}_6$.

\begin{mydef}[\cite{Laza}, Definition 2.16] \textnormal{ }
\begin{enumerate}[(i)]
\item
An element $\delta \in \Lambda_0$ such that $\delta^2=2$ is called a \emph{root}. We denote by $\Delta(\Lambda_0)$ the set of roots of the lattice $\Lambda_0$.
\item
An element $\delta \in \Lambda_0$ such that $\delta^2 = 6$ and $\langle \delta, \Lambda_0 \rangle \equiv 0 \mod 3$ is called a \emph{long root}. Denote by $\Delta_{long}(\Lambda_0)$ the set of longs roots of $\Lambda_0$.
\end{enumerate}
\end{mydef}

By (\cite{Laza}, Proposition 2.15), we have
$$\mathcal{D}_6 = \bigcup_{\delta \in \Delta(\Lambda_0)} \delta^{\perp} \subset \mathcal{D},$$
and
$$ \mathcal{D}_2 = \bigcup_{\delta \in \Delta_{long}(\Lambda_0)} \delta^{\perp} \subset \mathcal{D}.$$

Note that again $\mathcal{G}_n(\mathbb{Z}) \backslash \mathcal{D}^{\circ}$ is the complement of a union of special subvarieties of $Sh_{n,\mathbb{C}}$ corresponding to the sub-Shimura data attached to the group $$G_{\delta} := \{ g \in G \, | \, g\delta=\delta\}$$ for an element $ \delta \in \Delta(\Lambda_0) $ or $\delta \in \Delta_{long}(\Lambda_0)$.

\subsection{The image of the $p$-adic period map}

As before, let $\mathbb{X}$ denote a supersingular smooth cubic fourfold over $\bar{\mathbb{F}}_p$, and choose isometries $H_{\acute{e}t}^4(\mathbb{X}, \hat{\mathbb{Z}}^p) \cong \Lambda_{\hat{\mathbb{Z}}^p} $ and $(H^4_{cris}(\mathbb{X}/W)[\frac{1}{p}])^{\varphi=p^2} \cong \Lambda_{\mathbb{Q}_p} $ mapping the Chern class of $h^2$ to $\lambda$.
For $g \in G(\mathbb{A}_f^p)$ we introduce the lattice $$CH^2_{0,g} := CH^2_{0,\mathbb{Q}} \cap g\Lambda_{0,\hat{\mathbb{Z}}^p},$$ a $\mathbb{Z}[\frac{1}{p}]$-lattice which depends only on the class of $g$ in the double coset $I_{\phi}(\mathbb{Q}) \backslash G(\mathbb{A}_f^p) / K_n^p$.

\begin{mydef} \textnormal{ }
\begin{enumerate}[(i)]
\item
An element $\delta \in CH^2_{0,g}$ such that $\delta^2=2$ is called a \emph{root}. We denote by $\Delta(CH^2_{0,g})$ the set of roots of the lattice $CH^2_{0,g}$.
\item
An element $\delta \in CH^2_{0,g}$ such that $\delta^2 = 6$ and $\langle g^{-1} \delta, \Lambda_{0, \hat{\mathbb{Z}}^p} \rangle \equiv 0 \mod 3$ is called a \emph{long root}. Denote by $\Delta_{long}(CH^2_{0,g})$ the set of long roots of $CH^2_{0,g}$.
\end{enumerate}
\end{mydef}

Note that the expression $\langle g^{-1} \delta, \Lambda_{0, \hat{\mathbb{Z}}^p} \rangle \equiv 0 \mod 3$ makes sense as $g^{-1} \delta \in \Lambda_{0, \hat{\mathbb{Z}}^p}$.

\begin{mydef}
For $\delta \in \Delta(CH^2_{0,\mathbb{Q}})$, let $Z_{\delta} \subset X_{\mu}^G(b)(\bar{\mathbb{F}}_p)$ be the subset of those $ x \in G(\breve{\mathbb{Q}}_p)$ such that $x^{-1}(\delta) \in \Lambda_{0,W}$.
\end{mydef}

In analogy to the complex case we define
\begin{eqnarray*} Z_{6,g} := \bigcup_{\delta \in \Delta(CH^2_{0,g})} Z_{\delta}
& \textnormal{and} & Z_{2,g} := \bigcup_{\delta \in \Delta_{long}(CH^2_{0,g})} Z_{\delta}. \end{eqnarray*}
As before, the subset
$$ X^{\circ}_g(\bar{\mathbb{F}}_p) := X_{\mu}^G(b)(\bar{\mathbb{F}}_p) \setminus ( Z_{6,g} \cup Z_{2,g}) $$
can be seen as the set of $\bar{\mathbb{F}}_p $-points of an open subscheme $X_g^{\circ}$ of $X_{\mu}^G(b)$, and we set $\mathcal{M}^{\circ}_g := sp^{-1}(X^{\circ}_g)$.

\begin{thm}\label{cfuniformization}
For almost all $p$, the moduli space of smooth cubic fourfolds with supersingular reduction admits a uniformization
$$M_{cf, n, \breve{\mathbb{Q}}_p}^{ss} \cong \coprod_g \Gamma_g \backslash \mathcal{M}^{\circ}_g$$
of rigid spaces over $\breve{\mathbb{Q}}_p$, as well as of perfect schemes
$$M_{cf,ss,n, \bar{\mathbb{F}}_p}^{perf} \cong \coprod_g \Gamma_g \backslash X^{\circ}_g $$
over $\bar{\mathbb{F}}_p$.
\end{thm}

\begin{proof}[Sketch of proof]
All arguments given in Section \ref{spimage} can be applied accordingly:
for almost all primes $p$, the moduli space $M_{cf,ss,n, \mathbb{Z}_{(p)}}$ will coincide with the complement in $\mathscr{S}_{n, \mathbb{Z}_{(p)}}$ of a union of sub-Shimura varieties $\mathscr{S}_2 $ and $\mathscr{S}_6$. Applying Rapoport-Zink uniformization to $\mathscr{S}_2$ and $\mathscr{S}_6$ and tracing through the Hodge cycles defined by $\delta \in \Delta(\Lambda_0)$ or $\delta \in \Delta_{long}(\Lambda_0)$, we find that the corresponding crystalline and étale elements come from a root or a long root in $CH^2_{0,\mathbb{Q}}$, by identifying them with special endomor\-phisms and using (\cite{MP}, Theorem 5.14(2)). Conversely, any element $\delta \in \Delta(CH_{0,g})$ or $\delta \in \Delta_{long}(CH_{0,g})$ lifts to characteristic $0$ using (\cite{MP2}, Corollary 8.15).
\end{proof}

\addcontentsline{toc}{section}{References}
\bibliographystyle{alpha}
\bibliography{padicK3_arXiv}

\end{document}